\documentclass[final,hidelinks,onefignum,onetabnum]{siamonline220329}
\usepackage{graphicx} % Required for inserting images
\usepackage{amsmath,amssymb}
\usepackage{comment}
\usepackage{lipsum}
\usepackage{amsfonts}
\usepackage{graphicx}
\usepackage{epstopdf}
\usepackage{algorithmic}
\usepackage{enumitem}
\usepackage[sort]{cite}
\usepackage{nicematrix}
\usepackage{tikz}
\usetikzlibrary{arrows.meta}
\usepackage{pgfplots}
\usepgflibrary{shadings}
\usetikzlibrary{shadings}
\usepackage{overpic}
\ifpdf
  \DeclareGraphicsExtensions{.eps,.pdf,.png,.jpg}
\else
  \DeclareGraphicsExtensions{.eps}
\fi

\newcommand{\cof}{\textrm{cof}}

\newsiamremark{remark}{Remark}
\newsiamremark{hypothesis}{Hypothesis}
\newsiamremark{ex}{Example}
\crefname{hypothesis}{Hypothesis}{Hypotheses}
\newsiamthm{claim}{Claim}

\headers{Emil Graf and Alex Townsend}{Instability of Algebraic Rootfinding Methods}

\title{Numerical Instability of Algebraic Rootfinding Methods\thanks{\vspace{-1em} \funding{E.G. was supported by NSF GRFP (DGE-2139899). A.T. was supported by the Office of Naval Research under Grant Number N00014-23-1-2729 and NSF CAREER (DMS-2045646).}}}

\author{Emil Graf and Alex Townsend}

\usepackage{amsopn}
\DeclareMathOperator{\diag}{diag}
\begin{document}

\maketitle

% REQUIRED
\begin{abstract}
    We demonstrate that the most popular variants of all common algebraic multidimensional rootfinding algorithms are unstable by analyzing the conditioning of subproblems that are constructed at intermediate steps. In particular, we give multidimensional polynomial systems for which the conditioning of a subproblem can be worse than the conditioning of the original problem by a factor that grows exponentially with the number of variables.
\end{abstract}

% REQUIRED

% REQUIRED

\section{Introduction}

For over a decade, the scientific community has been searching for a numerically robust multivariate polynomial rootfinder  \cite{nakatsukasa2015bezout, noferini2016instability,telen2018normalform, telen2018normalform2, mourrain2021normalform,jonsson2005macaulay,plestenjak2016roots,rouillier1999rur,cox2005ag}. We are yet to find one. This paper demonstrates that the most popular variants of all common algebraic multidimensional rootfinding algorithms are exponentially unstable. This includes hidden-variable resultants (see \cite{noferini2016instability}), Gr\"obner bases (see \cref{sec:GB}), the rational univariate representation (see \cref{sec:RUR}), multiparameter eigenvalue problems (see \cref{sec:MEP}), normal form methods (see \cref{sec:MS}),  and Macaulay resultants (see \cref{sec:Mac}). We are stuck waiting for new ideas to emerge from algebraic geometry.

A multivariate rootfinder computes all the solutions of a polynomial system of the form:
\begin{equation} \label{eq:polysystem}
    \begin{pmatrix}
        p_1(x_1,\ldots,x_d) \\
        \vdots \\
        p_d(x_1,\ldots,x_d)
    \end{pmatrix} = \begin{pmatrix}
        0 \\
        \vdots \\
        0
    \end{pmatrix}, \quad (x_1,\ldots,x_d) \in \mathbb{C}^d,
\end{equation}
where $d \geq 2$, and $p_1,\ldots,p_d$ are polynomials in $x_1,\ldots,x_d$ with complex coefficients. We assume that the system \cref{eq:polysystem} has a finite number of simple roots with no roots at infinity. A multivariate rootfinder is a global rootfinder aimed at not missing any solutions to \cref{eq:polysystem}. Global rootfinders differ from local ones such as Newton--Raphson,  as they can find all the equilibrium points of a dynamical system or compute the global maximum of a polynomial.

Once we have a rootfinder that does not miss any solutions, we want the computed solutions to be accurate. Since, at some stage of a rootfinder, there must be a numerical calculation, as proved by the Abel--Ruffini theorem \cite{abel2012equations,ruffini1813riflessioni}, the scientific community focuses on well-conditioned problems. For example, consider Wilkinson's polynomial \cite{wilkinson1984polynomial} given by
$
    w(x) = (x-1)(x-2)\cdots(x-20).
$
If we increase the coefficient of $x^{19}$ by $2^{-23} \approx .0000001$, the root at $x=20$ moves to $x \approx 20.8$. Computing the roots of $w(x)$ given its coefficients is an ill-conditioned problem, i.e., small perturbations of its coefficients can cause large changes to its roots.  We are not interested in computing the roots of $w(x)$ as it is numerically a frivolous task\cite{wilkinson1984polynomial}. Instead, in scientific computing, we mainly focus on well-conditioned problems.

\subsection{Motivation for Eigenvalue-Based Approaches} Given a well-conditioned rootfinding problem, we would like to derive a stable algorithm to solve it. Roughly speaking, an algorithm is stable if it computes an accurate solution to well-conditioned problems (see \cref{subsec:cond}). The search for a stable algorithm for multivariate polynomial rootfinding is motivated by the existence of stable algorithms for many related problems (see \cref{tab:evp}). All the univariate problems, such as eigenproblems, univariate polynomial rootfinding, and matrix polynomial eigenproblems, have stable algorithms. Likewise, there are stable algorithms to solve linear systems of the form $A \mathbf{x} = \mathbf{b}$, which are multivariate.

\begin{table}[]
    \centering
    \caption{Types of linear equations,  rootfinding,  and eigenproblems}
    \resizebox{\columnwidth}{!}{
    \begin{tabular}{ccc}
    %\hline
    \multicolumn{3}{c}{Univariate Problems} \\
    \hline
    & Scalar Problems & Matrix Problems \\
    && \\
    Linear Problems & Basic Algebra & Generalized Eigenproblem\\
    & $ax = b$ & $A\mathbf{x} = \lambda B \mathbf{x}$ \\
    & \\
    Polynomial Problems & Polynomial Rootfinding & Polynomial Eigenproblem \\
    &$p(x) = 0$ & $P(x) \mathbf{v} = 0$ \\
    \\
    \hline\\
    %\hline
    \multicolumn{3}{c}{Multivariate Problems} \\
    \hline
    & Scalar Problems & Matrix Problems \\
    && \\
    Linear Problems & Linear System & Multivariate Eigenproblem\\
    & $A \mathbf{x} = \mathbf{b}$ & $W_i(\mathbf{x}) \mathbf{v}_i = 0,1 \leq i \leq d$ \\
    && \\
    Polynomial Problems & Multivariate Polynomial Rootfinding & Multivariate Polynomial Eigenproblem \\
    & $p_i(\mathbf{x}) = 0,1 \leq i \leq d$ & $P_i(\mathbf{x})\mathbf{v}_i = 0,1 \leq i \leq d$ \\
    && \\
    \hline
	\end{tabular}}
    \label{tab:evp}
\end{table}

One idea that works extremely well for univariate problems is to convert them into generalized eigenproblems. Instead of solving a rootfinding problem directly, one first constructs an eigenproblem whose eigenvalues match the roots. The companion matrix of a polynomial $p$ is an example of this, as its characteristic polynomial is $p$, so its eigenvalues are the roots of $p$ \cite[Chapt.~7]{golub2013matrix}.  One can solve the companion eigenproblem using an eigensolver,  which is one of the most reliable algorithms in scientific computing. For roots in $[-1,1]$, this is a stable algorithm for univariate polynomial rootfinding \cite{nakatsukasa2016stability,noferini2017stability}. Algebraic rootfinders attempt the same conversion, i.e., each algorithm converts \cref{eq:polysystem} into one or more generalized eigenvalue problems (GEPs).  For polynomial systems in \cref{eq:polysystem} in $d$ variables,  one usually constructs $d$ GEPs, the eigenvalues of which give the coordinates of each root. The Macaulay resultant method (see \cref{sec:Mac}) is an exception as it constructs a single GEP and extracts the roots from the eigenvectors, not eigenvalues. Analogously to the univariate case, one hopes that the eigenproblems are as well-conditioned as the original rootfinding problem, which is necessary for a stable algorithm. Unfortunately, this is not the case.%Unfortunately,  issues have been observed with all known eigenvalue-based rootfinders in $d>1$. 

%For the two parameter eigenproblem, we use code from \cite{plestenjak2024biroots}. 

%We generated orthogonal matrices randomly in matlab version 9.14.0.2239454 (R2023a) Update 1, setting the seed for the random number generator equal to $1$. All the data is the median over $1000$ trials for each value of $\sigma$ or $c$.

\pgfplotsset{every axis title/.style={at={(0.5,1)},above,yshift=-1pt},
			every axis x label/.style={at={(.5,-.15)},anchor=center},
        			every axis y label/.style={at={(-.12,.5)},rotate=90,anchor=center}}

\begin{figure}
  \centering
  
  \vspace{-.3cm}
  
  \begin{Overpic}{
  \begin{tikzpicture}
    \begin{axis}[
      title={Bezout Resultant},
      height=2.3in,width=3.1in,xmin =-8,xmax=0,ymin=0,ymax=16,
      xlabel={$\log (\sigma)$}, 
      ylabel={Digits of Accuracy}
      ]
      \addplot [color=teal,thick,dotted,mark = none] table[col sep=comma,x = {x}, y ={y3}] {D2data.dat};
      \addplot [color=red,very thick,mark = none] table[col sep=comma,x = {x}, y ={y}] {D2data.dat};
      \addplot [color=blue,very thick,only marks] table[col sep=comma,x = {x}, y ={S5}] {D2data.dat};
       \end{axis}
  \end{tikzpicture}}
  \put(6,69){\textbf{(a)}}
  \put(32,47.5){\rotatebox{18.5}{\textcolor{teal}{Stable Performance}}}
  \put(21,35){\rotatebox{-3}{\textcolor{blue}{Practical}}}
  \put(41,35){\rotatebox{34}{\textcolor{blue}{Performance}}}
  \put(50,32){\rotatebox{34}{\textcolor{red}{Our Theory}}}
  \end{Overpic}
  \begin{Overpic}{
  \begin{tikzpicture}
    \pgfplotsset{every axis title/.style={at={(0.5,1)},above,yshift=-3pt}}
    \begin{axis}[
      title={Sylvester Resultant},
      height=2.3in,width=3.1in,xmin =-8,xmax=0,ymin=0,ymax=16,
      xlabel={$\log (\sigma)$} %ylabel={Digits of Accuracy},
      ]
      \addplot [color=teal,thick,dotted,mark = none] table[col sep=comma,x = {x}, y ={y3}] {D2data.dat};
      \addplot [color=red,very thick,mark = none] table[col sep=comma,x = {x}, y ={y}] {D2data.dat};
      \addplot [color=blue,very thick,only marks] table[col sep=comma,x = {x}, y ={S6}] {D2data.dat};
       \end{axis}
  \end{tikzpicture}}
  \put(0,73){\textbf{(b)}}
  \put(28,50){\rotatebox{19}{\textcolor{teal}{Stable Performance}}}
  \put(15,37){\rotatebox{-3}{\textcolor{blue}{Practical}}}
  \put(37,37.5){\rotatebox{34}{\textcolor{blue}{Performance}}}
  \put(43,31){\rotatebox{34}{\textcolor{red}{Our Theory}}}
  \end{Overpic}
  
   \vspace{-.1cm}
  
  \begin{Overpic}{
  \begin{tikzpicture}
    \begin{axis}[
      title={Gr\"obner Basis Elimination},
      height=2.3in,width=3.1in,xmin =-8,xmax=0,ymin=0,ymax=16,
      xlabel={$\log (\sigma)$}, 
      ylabel={Digits of Accuracy}
      ]
      \addplot [color=teal,thick,dotted,mark = none] table[col sep=comma,x = {x}, y ={y3}] {D2data.dat};
      \addplot [color=red,very thick,mark = none] table[col sep=comma,x = {x}, y ={y5}] {D2data.dat};
      \addplot [color=blue,very thick,only marks] table[col sep=comma,x = {x}, y ={S7}] {D2data.dat};
       \end{axis}
  \end{tikzpicture}}
   \put(6,69){\textbf{(c)}}
  \put(32,47.5){\rotatebox{18.5}{\textcolor{teal}{Stable Performance}}}
  \put(33,29){\rotatebox{0}{\textcolor{blue}{Practical}}}
  \put(53,31){\rotatebox{44}{\textcolor{blue}{Performance}}}
  \put(65,29){\rotatebox{44}{\textcolor{red}{Our Theory}}}
  \end{Overpic}
  \begin{Overpic}{
  \begin{tikzpicture}
  \pgfplotsset{every axis title/.style={at={(0.5,1)},above,yshift=-3pt}}
    \begin{axis}[
      title={Rational Univariate Representation},
      height=2.3in,width=3.1in,xmin =-8,xmax=0,ymin=0,ymax=16,
      xlabel={$-\log (c)$} %ylabel={Digits of Accuracy},
      ]
      \addplot [color=teal,thick,dotted,mark = none] table[col sep=comma,x = {x2}, y ={y4}] {D2data.dat};
      \addplot [color=red,very thick,mark = none] table[col sep=comma,x = {x2}, y ={y2}] {D2data.dat};
      \addplot [color=blue,very thick,only marks] table[col sep=comma,x = {x2}, y ={S4}] {D2data.dat};
       \end{axis}
  \end{tikzpicture}}
  \put(0,73){\textbf{(d)}}
  \put(27.5,49.5){\rotatebox{19}{\textcolor{teal}{Stable Performance}}}  
  \put(31,31){\rotatebox{2}{\textcolor{blue}{Practical}}}
  \put(52.2,32.5){\rotatebox{45}{\textcolor{blue}{Performance}}}
  \put(65,31.2){\rotatebox{45}{\textcolor{red}{Our Theory}}}
  \end{Overpic}
  
  \vspace{-.1cm}
  
  \begin{Overpic}{
  \begin{tikzpicture}
  \pgfplotsset{every axis title/.style={at={(0.5,1)},above,yshift=-3pt}}
    \begin{axis}[
      title={Two-Parameter Eigenproblem},
      height=2.3in,width=3.1in,xmin =-8,xmax=0,ymin=0,ymax=16,
      xlabel={$\log (\sigma)$}, ylabel={Digits of Accuracy}
      ]
      \addplot [color=teal,thick,dotted,mark = none] table[col sep=comma,x = {x}, y ={y3}] {D2data.dat};
      \addplot [color=red,very thick,mark = none] table[col sep=comma,x = {x}, y ={y}] {D2data.dat};
      \addplot [color=blue,very thick,only marks] table[col sep=comma,x = {x}, y ={S2}] {D2data.dat};
       \end{axis}
  \end{tikzpicture}}
  \put(6,69){\textbf{(e)}}
  \put(32,47.5){\rotatebox{18.5}{\textcolor{teal}{Stable Performance}}}
  \put(19,34.5){\rotatebox{3}{\textcolor{blue}{Practical}}}
  \put(39,36.5){\rotatebox{34}{\textcolor{blue}{Performance}}}
  \put(52,33){\rotatebox{34}{\textcolor{red}{Our Theory}}}
  \end{Overpic}
  \begin{Overpic}{
  \begin{tikzpicture}
    \begin{axis}[
      title={Normal Form Method},
      height=2.3in,width=3.1in,xmin =-8,xmax=0,ymin=0,ymax=16,
      xlabel={$\log (\sigma)$} %ylabel={Digits of Accuracy},
      ]
      \addplot [color=teal,thick,dotted,mark = none] table[col sep=comma,x = {x}, y ={y3}] {D2data.dat};
      \addplot [color=red,very thick,mark = none] table[col sep=comma,x = {x}, y ={y}] {D2data.dat};
      \addplot [color=blue,very thick,only marks] table[col sep=comma,x = {x}, y ={S1}] {D2data.dat};
       \end{axis}
  \end{tikzpicture}}
  \put(0,73){\textbf{(f)}}
  \put(28,50){\rotatebox{19}{\textcolor{teal}{Stable Performance}}}
  \put(17,36){\rotatebox{-4}{\textcolor{blue}{Practical}}}
  \put(39,35.5){\rotatebox{34}{\textcolor{blue}{Performance}}}
  \put(44,30){\rotatebox{34}{\textcolor{red}{Our Theory}}}
  \end{Overpic}
  
  \vspace{-.1cm}
  
  \begin{Overpic}{
  \begin{tikzpicture}
  \pgfplotsset{every axis title/.style={at={(0.5,1)},above,yshift=-3pt}}
    \begin{axis}[
      title={Macaulay Resultant},
      height=2.3in,width=3.1in,xmin =-8,xmax=0,ymin=0,ymax=16,
      xlabel={$\log (\sigma)$}, ylabel={Digits of Accuracy}
      ]
      \addplot [color=teal,thick,dotted,mark = none] table[col sep=comma,x = {x}, y ={y3}] {D2data.dat};
      \addplot [color=red,very thick,mark = none] table[col sep=comma,x = {x}, y ={y}] {D2data.dat};
      \addplot [color=blue,very thick,only marks] table[col sep=comma,x = {x}, y ={S3}] {D2data.dat};
       \end{axis}
  \end{tikzpicture}}
  \put(6,69){\textbf{(g)}}
  \put(32,47.5){\rotatebox{18.5}{\textcolor{teal}{Stable Performance}}}
  \put(22,32.5){\rotatebox{3}{\textcolor{blue}{Practical}}}
  \put(42,34.5){\rotatebox{34}{\textcolor{blue}{Performance}}}
  \put(52,32){\rotatebox{34}{\textcolor{red}{Our Theory}}}
  \end{Overpic}
  
  \vspace{-.6cm}
  
  \caption{Performance on a bivariate version of the system in \cref{ex:dev} for all methods except for Gr\"obner basis elimination (\cref{ex:GBdev}) and the rational univariate representation (\cref{ex:Hypercube}), with the root shifted to $(\frac{1}{3},\frac{1}{3})$. The deviation of the practical performance from our theory is explained by the extreme proximity of the roots when $\sigma$ is very small, which indicates that we should not expect conditioning analysis to be a good predictor. The methods are extremely inaccurate when $\sigma$ is small.}
  \label{fig:devex}
\end{figure}
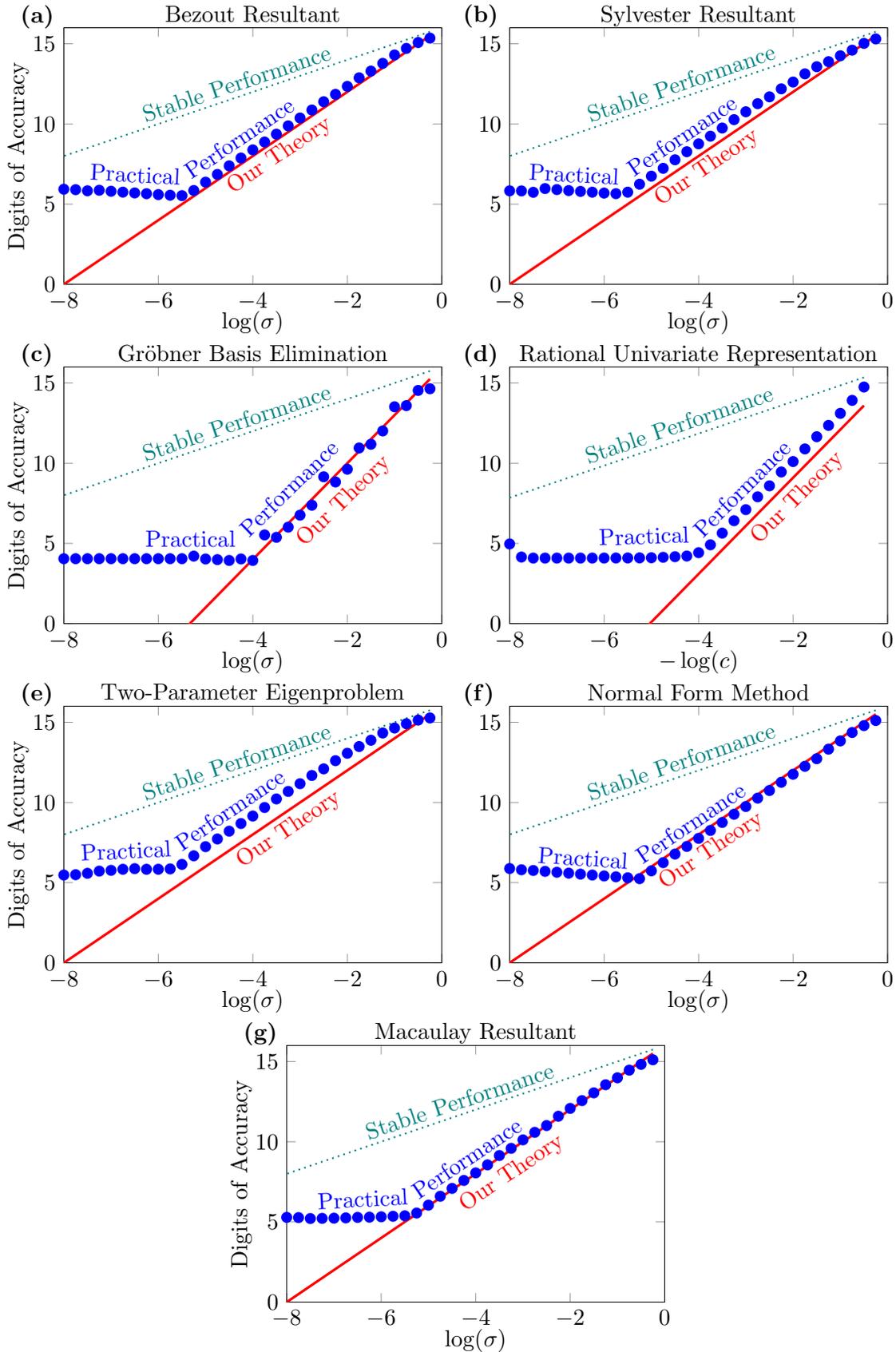

\cref{fig:devex} demonstrates the instability of bivariate algebraic rootfinders. Far from achieving any global stability, the most popular algorithms are not even stable for finding a single root at $(\frac{1}{3},\frac{1}{3})$, from simple systems generated by two quadratic polynomials.\footnote{For hidden-variable resultants, we implement the methods in \cite{nakatsukasa2015bezout, noferini2016instability}. For the two-parameter eigenproblem, we use the code from \cite{plestenjak2024biroots}.  For normal form methods, we implement the null space method of \cite{telen2018normalform, telen2018normalform2, mourrain2021normalform}. We generate orthogonal matrices randomly in MATLAB R2023a,  setting the seed for the random number generator equal to $1$. All the data is the median over $1000$ trials for each value of $\sigma$ or $c$.}

The first analysis of the instability of algebraic rootfinders appears in \cite{nakatsukasa2015bezout,noferini2016instability}, which consider the Sylvester and Cayley/Bezout hidden-variable resultant methods.  In particular, the authors study the following ``devastating'' example.
\begin{ex} \label{ex:dev}
    Let $Q$ be a $d \times d$ orthogonal matrix, $\sigma>0$, and consider \cref{eq:polysystem} with
    $$
    p_i(x_1,\ldots,x_d) = x_i^2 + \sigma \sum_{j=1}^d q_{ij} x_j, \quad 1 \leq i \leq d,
    $$
    where $q_{ij}$ is the $(i,j)$ entry of $Q$. The system has a root at $(0,\ldots,0)$.
\end{ex}
By a conditioning analysis, one should expect to find the root at $(0,\ldots,0)$ to within $\approx \sigma \mathbf{u}$, where $\mathbf{u}$ is the unit roundoff on a computer. However, it has been shown that Sylvester resultants can only achieve $\approx \sigma^{-2} \mathbf{u}$ when $d=2$ and Cayley can only achieve $\approx \sigma^{-d} \mathbf{u}$ \cite{noferini2016instability}. The eigenproblems constructed by these methods can be far more sensitive to perturbations than the original rootfinding problem, which is a hallmark of an unstable algorithm.

Each algebraic rootfinder transforms \cref{eq:polysystem} into one or more GEPs, either directly or by way of a univariate rootfinding problem. For each method, we show that either a constructed eigenproblem or an intermediate univariate rootfinding problem can be ill-conditioned by a factor that depends exponentially on $d$.  %We formalize this by studying the condition number of these subproblems. Each algorithm constructs an ill-conditioned subproblem relative to the original problem,  making them all unstable.
%\begin{figure} \label{fig:instabilityflowchart}
%\begin{center}
%\begin{tikzpicture}
%\draw
%	(0,0) node [fill = blue!70!black!30,draw,double,rounded corners] {Goal:}
%	(4,0) node [fill = green!70!black!30,draw,double,rounded corners, text width=4cm] {Well-Conditioned Rootfinding Problem}
%	(11,0) node [fill = green!70!black!30,draw,double,rounded corners, text width=5.5cm] {Well-Conditioned Eigenvalue or \\ Univariate Rootfinding Problem}
%	[line width=1mm,-{Stealth[length=10mm, open]}]  (6.2,0) -> (8.1,0);
%\draw
%	(0,-1.5) node [fill = blue!70!black!30,draw,double,rounded corners] {In Practice:}
%	(4,-1.5) node [fill = green!70!black!30,draw,double,rounded corners, text width=4cm] {Well-Conditioned Rootfinding Problem}
%	(11,-1.5) node [fill = red!70!black!30,draw,double,rounded corners, text width=5.5cm] {Ill-Conditioned Eigenvalue or \\ Univariate Rootfinding Problem}
%	[line width=1mm,-{Stealth[length=10mm, open]}]  (6.2,-1.5) -> (8.1,-1.5);
%\end{tikzpicture}
%\caption{Numerical instability of multidimensional rootfinding algorithms.}
%\end{center}
%\end{figure}
Our results are summarized in \cref{tab:summary}, where $\mathbf{x} = (x_1,\ldots,x_d)$ and $J(\mathbf{x}^*)$ denotes the Jacobian at the root $\mathbf{x}^*$; other notation is explained in the corresponding section, or for hidden-variable resultants in \cite{nakatsukasa2015bezout,noferini2016instability}. The last column of \cref{tab:summary} is the ratio between a subproblem's condition number and that of the original root.   By identifying a source of instability,  we hope practitioners can focus on circumventing it.

\begin{table}[]
    \centering
    \caption{The condition number of subproblems constructed by popular algebraic rootfinders.}
    \resizebox{\columnwidth}{!}{
    \begin{tabular}{cccc}
\hline
& & & \\
    Method & \begin{tabular}{c} Condition Number \\ of Subproblem \end{tabular} & Devastating Example & \begin{tabular}{c} Condition \\ Number Ratio \end{tabular} \\
     &  & & \\
     \hline
    & & & \\
    Cayley Resultant & $\kappa_{\text{eig}}(x_i^*) \geq 
    \frac{||V(\mathbf{x}^*)||_2||W(\mathbf{x}^*)||_2}{\lvert \det(J(\mathbf{x}^*)) \rvert}
    $ & $p(\mathbf{x}) = \mathbf{x}^2 + \sigma Q \mathbf{x}$ & $\frac{\kappa_{\text{eig}}}{\kappa_{\text{root}}} \geq \sigma^{-d+1}$ \\
    (See \cite{noferini2016instability}) & & &\\
    & & & \\
    Sylvester Resultant (2D) & $\kappa_{\text{eig}}(x_i^*) \geq 
    \frac{||v(\mathbf{x}^*)||_2||w(\mathbf{x}^*)||_2}{ \lvert \det(J(\mathbf{x}^*)) \rvert}
    $ & $p(\mathbf{x}) = \mathbf{x}^2 + \sigma Q \mathbf{x}$ & $\frac{\kappa_{\text{eig}}}{\kappa_{\text{root}}} \geq \sigma^{-1}$ \\
    (See \cite{noferini2016instability})  & & &\\
    & & & \\
    Gr\"obner Basis Elimination & (See \cref{sec:GB}) & $p_i(\mathbf{x}) = x_i^2 - \sigma x_{i+1}^2 $ & $\frac{\kappa_{\text{uni}}}{\kappa_{\text{root}}} = \sigma^{-2^d+2}$ \\
    (See \cref{sec:GB}) &  & &\\
    & & & \\
    Rational Univariate & (See \cref{sec:RUR}) & $p(\mathbf{x}) = A \mathbf{x}^2 -\frac{A}{c^2d} \mathbf{1}$ & $\frac{\kappa_{\text{uni}}}{\kappa_{\text{root}}} \geq \frac{||A||_2}{\sqrt{d}}\left( \frac{c}{2} \right)^{2^d-2}$ \\
    (See \cref{sec:RUR}) & & & \\
     & & & \\
    Multiparameter Eigenvalue & $
    \kappa_{\text{eig}}(x_i^*) \geq \frac{ \prod_{k=1}^d \prod_{j=1}^{n_k-1} \sigma_j^{(k)}(\mathbf{x}^*)}{ \lvert \det(J(\mathbf{x}^*)) \rvert}
    $ & $p(\mathbf{x}) = \mathbf{x}^2 + \sigma P \mathbf{x}$ &  $\frac{\kappa_{\text{eig}}}{\kappa_{\text{root}}} \geq \sigma^{-d+1}$\\
    (See \cref{sec:MEP}) & & & \\
    & & &\\
    Normal Form  & $\kappa_{\text{eig}}(x_i^*) \geq
    \frac{||[\det(Q)]_{\mathcal{B}}||_2||\mathcal{B}(\mathbf{x}^*)||_2}{\lvert \det(J(\mathbf{x}^*)) \rvert}
    $ & $p(\mathbf{x})  = \mathbf{x}^2 + \sigma P \mathbf{x}$ & $\frac{\kappa_{\text{eig}}}{\kappa_{\text{root}}} \geq \sigma^{-d+1}$\\
     (See \cref{sec:MS}) & & & \\
     & & & \\
    Macaulay Resultant  & $\kappa_{\text{eig}}(\lambda^*) \geq
    \frac{||[\det(Q)]_{\mathcal{B}}||_2||V(\mathbf{x}^*)||_2}{\lvert \det(J(\mathbf{x}^*)) h(\mathbf{x}^*) \rvert}
    $ & $p(\mathbf{x})  = \mathbf{x}^2 + \sigma P \mathbf{x}$ &  $\frac{\kappa_{\text{eig}}}{\kappa_{\text{root}}} \geq \sigma^{-d+1}$\\
     (See \cref{sec:Mac}) & & & \\
    \hline
\end{tabular}}
    \label{tab:summary}
\end{table}

%The results for Sylvester and Cayley hidden variable resultants appear in \cite{noferini2016instability}. Here we give a similar analysis for the remaining methods. We begin by giving examples to show the instability of the Gr\"obner basis (\cref{sec:GB}) and rational univariate representation (\cref{sec:RUR}) methods. We then give complete characterizations of the condition numbers of the eigenproblems constructed in the multiparameter eigenvalue (\cref{sec:MEP}), M\"oller--Stetter (\cref{sec:MS}), and Macaulay resultant methods (\cref{sec:Mac}), along with examples to show the exponential ill-condtioning. 

%We believe that there are the principal algebraic methods in the literature; as such our analysis demonstrates instability that will arise with the eigenvalue-based approach regardless of which method is chosen. This theoretical analysis helps to explain significant issues that arise in practice.

%There are other potential sources of instability in these algorithms, such as the construction of the eigenproblem. We are not as concerned about these steps as they can in principle be done symbolically. Solving the eigenproblem must be done numerically, so we believe the instability is unavoidable for current methods. Moreover, this instability exists for any reasonable polynomial basis and can not be overcome by any numerical technique we know.

\section{Background}

We begin with eigenproblems, conditioning, and algebraic geometry. 

\subsection{Generalized Eigenproblems and Matrix Polynomials} Algebraic rootfinders convert \cref{eq:polysystem} into one or more GEPs. A GEP is given by $A \mathbf{x} = \lambda B \mathbf{x}$, where $A,B \in \mathbb{C}^{n \times n}$, and one seeks to find all the eigenvalue-eigenvector pairs $(\lambda,\mathbf{x})$ that satisfy $A \mathbf{x} = \lambda B \mathbf{x}$.  When $B$ is the identity matrix, a GEP simplifies to the standard eigenproblem $A \mathbf{x} = \lambda \mathbf{x}$.

Note that $A \mathbf{x} = \lambda B \mathbf{x}$ if and only if $(A-\lambda B) \mathbf{x} = 0$, so the eigenvalues are the values of $\lambda$ for which $\det(A-\lambda B) = 0$. A nonzero vector $\mathbf{x}$ satisfying $(A-\lambda B) \mathbf{x} = 0$ is known as a right eigenvector, while a nonzero vector $\mathbf{y}$ satisfying $\mathbf{y}^{\top}(A-\lambda B) = 0$, where $\mathbf{y}^{\top}$ denotes the transpose of the vector $\mathbf{y}$, is known as a left eigenvector.  %The eigenvectors are key to understanding the instability introduced into a rootfinder by an eigenproblem (see \cref{subsec:cond}).

The form $A-\lambda B$ naturally generalizes in $d$ variables to expressions of the form $V_0 + \sum_{i=1}^d x_i V_i$, where $V_i \in \mathbb{C}^{n \times n}$ for some integer $n$; these are known as linear matrix polynomials, and they form the basis for multiparameter eigenproblems (see \cref{sec:MEP}).

\subsection{Condition Numbers} \label{subsec:cond}
For a function $g$, the normwise absolute condition number is a measurement of how much the output changes for small changes in input around $x^*$, i.e.,
$$
\kappa(x^*) = \lim_{\epsilon \to 0} \sup_{||\delta x||_a < \epsilon} \frac{||g(x^*+\delta x) - g(x^*)||_b}{||\delta x||_a},
$$
where $||\cdot||_a$ and $||\cdot||_b$ denote norms on the input and output spaces. The condition number is unique once a norm is chosen for both the input and output, and tells you how much a solution can be perturbed given a perturbation of the input.   A stable algorithm should compute a solution with error at most $\approx \kappa(x^*) \mathbf{u}$,  where $\mathbf{u}$ is the unit roundoff or machine epsilon. 

The absolute condition number of a root $\mathbf{x}^* = (x_1^*,\ldots,x_d^*)$ of \cref{eq:polysystem} with respect to the spectral norm is \cite{nakatsukasa2015bezout}
\begin{equation} \label{eq:Jacobian}
\kappa_{\text{root}}(\mathbf{x}^*) = ||(J(\mathbf{x}^*))^{-1}||_2, \quad
J(\mathbf{x}^*) =\begin{pmatrix}
    \frac{\partial p_1}{\partial x_1}(\mathbf{x}^*) & \cdots & \frac{\partial p_1}{\partial x_d}(\mathbf{x}^*)\\
    \vdots & \ddots & \vdots \\
    \frac{\partial p_d}{\partial x_1}(\mathbf{x}^*) & \cdots & \frac{\partial p_d}{\partial x_d}(\mathbf{x}^*)
\end{pmatrix}.
\end{equation}
The matrix in \cref{eq:Jacobian} is the Jacobian matrix of \cref{eq:polysystem}, which specializes in the univariate case to $\kappa_{\text{uni}}(x^*) = \lvert (p_1'(x^*))^{-1} \rvert$. For Gr\"obner basis elimination (see \cref{sec:GB}) and the rational univariate representation (see \cref{sec:RUR}), we use this condition number to analyze the instability introduced by an intermediate univariate rootfinding problem.

Algebraic rootfinders convert~\cref{eq:polysystem} into an eigenproblem. For multiparameter eigenproblems (see \cref{sec:MEP}), normal form methods (see \cref{sec:MS}), and Macaulay resultants (see \cref{sec:Mac}), we are interested in the condition number of the constructed eigenproblem.  We define the normwise absolute condition number for $A \mathbf{x} = \lambda B \mathbf{x}$ at eigenvalue $\lambda^*$ as
$$
\kappa_{\text{eig}}(\lambda^*) \hspace{-.1cm} = \hspace{-.1cm} \limsup_{\epsilon \to 0} \hspace{-.1cm} \bigg\{ \hspace{-.1cm} \frac{||\Delta \lambda||_2}{\epsilon} \hspace{-.1cm} : \hspace{-.1cm}  \left((A \hspace{-.1cm} +\hspace{-.1cm} \Delta A) \hspace{-.1cm} - \hspace{-.1cm} (\lambda^*+\Delta \lambda)(B\hspace{-.1cm} +\hspace{-.1cm}\Delta B)\right) (\mathbf{x}+\Delta \mathbf{x}) = 0, ||\Delta A||_2, ||\Delta B||_2 \leq \epsilon  \hspace{-.1cm} \bigg\}.
$$
A similar definition is given in \cite{noferini2016instability}. It is immediate from  \cite{tisseur1998conditioning} that
\begin{equation} \label{eq:geneigconditioning}
\kappa_{\text{eig}}(\lambda^*) = \frac{||\mathbf{y}||_2||\mathbf{x}||_2}{|\mathbf{y}^{\top}B\mathbf{x}|} (1 + |\lambda^*|),
\end{equation}
where $\mathbf{y}$ and $\mathbf{x}$ are the left and right eigenvectors for the eigenvalue $\lambda^*$. In this formulation, the dependence of the condition number on $A$ is encapsulated in the left and right eigenvectors.

A stable rootfinder does not construct any subproblem that is more ill-conditioned than the original rootfinding problem. Unfortunately, we find that every algebraic rootfinder constructs either a univariate rootfinding problem or a GEP with a condition number worse than the original system's by a factor that depends exponentially on $d$.

%There is also an analogous definition of the relative condition number, and, for matrix eigenproblems, of entrywise instead of absolute condition numbers. To quote N.J. Higham \cite[p. 56]{higham2008functions}, ``Usually, it is the relative condition number that is of interest, but it is more convenient to state results for the absolute condition number." We have found the same for our analysis, and so while in practice entrywise relative perturbations are of interest we state all results for the normwise absolute condition number and support our theoretical analysis with numerical experiments.

\subsection{Zero-Dimensional Ideals and Varieties} The polynomial system of $d$ polynomials in $d$ variables in \cref{eq:polysystem} generates a polynomial ideal $\langle p_1,\ldots,p_d \rangle = \{f_1p_1+\ldots+f_dp_d : f_i \in \mathbb{C}[x_1,\ldots,x_d] \}$. The system, or equivalently the ideal, is zero-dimensional if the variety $\mathcal{V}(p_1,\ldots,p_d) = \{(x_1,\ldots,x_d) \in \mathbb{C}^d: p_i(x_1,\ldots,x_d) = 0, 1 \leq i \leq d\}$ consists of finitely many points. A root of a polynomial system is simple if the Jacobian $J(\mathbf{x}^*)$ is invertible. Suppose all the roots of \cref{eq:polysystem} are simple. In that case, being zero-dimensional is equivalent to the quotient $\mathbb{C}[x_1,\ldots,x_d]/ \langle p_1,\ldots,p_d \rangle$ being a finite-dimensional vector space, whose dimension equals the number of roots. Throughout, our polynomial systems have only simple roots.

The equivalence in the last paragraph can be illustrated through the Lagrange interpolant basis for $\mathbb{C}[x_1,\ldots,x_d]/ \langle p_1,\ldots,p_d \rangle$, which is important to later analysis. We call a polynomial $q_{\mathbf{x}}$ a Lagrange interpolant for the system  at $\mathbf{x}$ if $q_{\mathbf{x}}(\mathbf{x}) \neq 0$ and $q_{\mathbf{x}}(\mathbf{x}') = 0$ for all $\mathbf{x}' \in \mathcal{V}(p_1,\ldots,p_d)$, whenever $\mathbf{x}' \neq \mathbf{x}$. If a system is zero-dimensional with simple roots, then the Lagrange interpolants with $q_{\mathbf{x}}(\mathbf{x}) = 1$ are a basis for $\mathbb{C}[x_1,\ldots,x_d]/ \langle p_1,\ldots,p_d \rangle$.

The radical of a polynomial ideal is denoted by $\langle p_1,\ldots,p_d \rangle$ and defined as $\sqrt{\langle p_1,\ldots,p_d \rangle} = \{ f \in \mathbb{C}[x_1,\ldots,x_d] : f^m \in \langle p_1,\ldots,p_d \rangle, m \in \mathbb{Z}_{\geq 0} \}.$ An ideal is radical if $\langle p_1,\ldots,p_d \rangle = \sqrt{\langle p_1,\ldots,p_d \rangle}$. If $f^m \in \langle p_1,\ldots,p_d \rangle$ then $f^m = f = 0$ in $\mathbb{C}[x_1,\ldots,x_d]/ \langle p_1,\ldots,p_d \rangle$, so $f \in \langle p_1,\ldots,p_d \rangle$. Thus, a zero-dimensional polynomial system with simple roots generates a radical ideal, so $\langle p_1,\ldots,p_d \rangle$ is radical throughout.

In addition, we assume that the system in \cref{eq:polysystem} has no roots at infinity; this is generic, and any system with roots at infinity can be transformed via a random linear change of variables into one with all finite roots with probability one.

\section{Gr\"obner Basis Elimination} \label{sec:GB}

 A Gr\"obner basis is a specific type of generating set for a polynomial ideal that allows one to deduce many important properties of the ideal. They are a popular technique for multivariate rootfinding \cite[Chapt. 2]{cox2005ag}. For concreteness, consider solving \cref{eq:polysystem} using a Gr\"obner basis with respect to the lexicographical order. Since \cref{eq:polysystem} has a zero-dimensional variety, the Gr\"obner basis contains a univariate polynomial $g \in \langle p_1,\ldots,p_d \rangle$ that generates the elimination ideal $\langle p_1,\ldots,p_d \rangle \cap \mathbb{C}[x_i]$ and as such determines the $x_i$ coordinates of the roots. However, this univariate polynomial can have extremely ill-conditioned roots, as demonstrated in the following example.

\begin{ex} \label{ex:GBdev}
Let $\sigma > 0$ and consider the system in $\cref{eq:polysystem}$ with polynomials
$$
p_i(x_1,\ldots,x_d) = x_i^2 -  \sigma x_{i+1}, \quad 1 \leq i \leq d-1,  \quad p_d(x_1,\ldots,x_d) = x_d^2 - \sigma x_1.
$$
Expanding $p_{i-1}(x_1,\ldots,x_d) \cdot (x_{i-1}^2+\sigma x_i)$ shows that $x_i^2 = \sigma^{-2}x_{i-1}^4$ in the quotient space $\mathbb{C}[x_1,\ldots,x_d]/ \langle p_1,\ldots,p_d \rangle$. Thus $x_{i} = \sigma^{-1}x_{i-1}^2 =\sigma^{-1}(\sigma^{-2}x_{i-2}^4) = \sigma^{-1}(\sigma^{-2}(\sigma^{-2}x_{i-3}^4)^2) = \cdots = \sigma^{-\sum_{k=0}^{d-1} 2^k} x_{i}^{2^d} = \sigma^{-2^d+1}x_{i}^{2^d}$. Let $g = x_{i}^{2^d} - \sigma^{2^d-1} x_i$. We know that $g \in \langle p_1,\ldots,p_d \rangle$ and has $2^d$ distinct roots, so $g$ generates the elimination ideal $\langle p_1,\ldots,p_d \rangle \cap \mathbb{C}[x_i]$. Due to the system's symmetry,  the same univariate problem is generated for any ordering of the variables. 

The root of the polynomial system at $(0,\ldots,0)$ has absolute condition number given by $\kappa_{\text{root}}(0,\ldots,0) = \sigma^{-1}$, but the root of $g$ at $0$ has absolute condition number $\kappa_{\text{uni}}(0) = \lvert g'(0)^{-1} \rvert = \sigma^{-2^d+1}$. Thus, the Gr\"obner basis elimination method has generated a univariate rootfinding problem that is exponentially ill-conditioned relative to the original system, so the method is unstable. This instability is independent of the polynomial basis chosen and the order of the variables, i.e., it can not be avoided by representing $g$ in a non-monomial basis or by changing the monomial order.
\end{ex}

The instability is illustrated in two variables in \cref{fig:devex} \textbf{(c)} for varying values of $\sigma$, and in \cref{fig:devexhighdimGB} for $\sigma = \frac{1}{2}$ and varying values of $d$. 

\begin{figure}
  \centering
  \begin{Overpic}{
  \begin{tikzpicture}
   %\pgfplotsset{set layers}
    \begin{axis}[
      title={Gr\"obner Basis Elimination},height=2.3in,width=3.1in,xmin =1,xmax=8,ymin=0,ymax=16,
      %legend pos = north east,
      xlabel={$d$}, ylabel={Digits of Accuracy},
      ]
      \addplot [color=teal,thick,dotted,mark = none] table[col sep=comma,x = {xh}, y ={yhtGB}] {HDdata.dat};
      \addplot [color=red,very thick,mark = none] table[col sep=comma,x = {xh}, y ={yhGB}] {HDdata.dat};
      \addplot [color=blue,very thick,only marks] table[col sep=comma,x = {xh}, y ={S8}] {HDdata.dat};
      %\addplot [color=violet,thick,dotted] table[col sep=comma,x = {x}, y ={D2}] {D2data.dat};
       \end{axis}
       %\begin{pgfonlayer}{axis background}
       %\filldraw [fill=gray!60] (1.7,0) rectangle (2.45,4);
       %\end{pgfonlayer}
  \end{tikzpicture}}
  \put(42,61){\rotatebox{0}{\textcolor{teal}{Stable Performance}}}
  \put(25,58){\rotatebox{-28}{\textcolor{blue}{Practical}}}
  \put(45,46){\rotatebox{-66}{\textcolor{blue}{Performance}}}
  \put(53,50){\rotatebox{-50}{\textcolor{red}{Our}}}
  \put(59,42){\rotatebox{-65}{\textcolor{red}{Theory}}}
  \end{Overpic}
   \label{fig:devexhighdimGB}
  
   \vspace{-.6cm}
   
   \caption{Performance of Gr\"obner basis elimination on \cref{ex:GBdev} for $d \geq 2$ and $\sigma = \frac{1}{2}$. We plot the practical performance against the theoretical performance of a stable algorithm and the prediction given by the analysis of \cref{ex:GBdev}. The deviation of the practical performance from our theory is explained by the extreme proximity of the roots for large values of $d$.}
\end{figure}
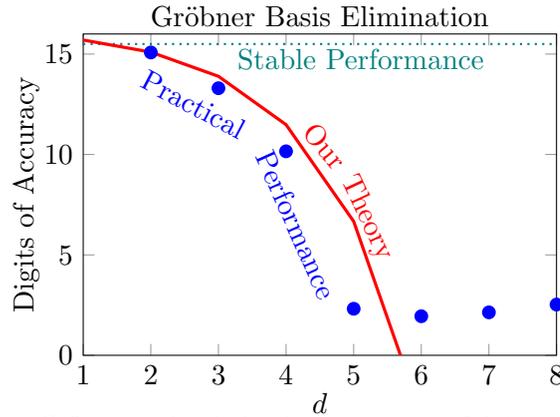

%An alternate possibility, as laid out in \cite[Chapter 2]{cox2005ag}, is to use a Gr\"obner basis to calculate M\"oller--Stetter matrices; this gives an alternate construction of the matrices, but the conditioning analysis is the same as presented in \cref{sec:MS}.

\section{Rational Univariate Representation} \label{sec:RUR}
The rational univariate representation projects a multivariate rootfinding problem to a univariate one \cite{rouillier1999rur}. A projection is determined by a polynomial $t \in \mathbb{C}[x_1,\ldots,x_d]$ such that $t(\mathbf{x}) \neq t(\mathbf{x}')$ for distinct roots $\mathbf{x}$ and $\mathbf{x}'$ of \cref{eq:polysystem}; we call this a separating polynomial. Given a separating polynomial, one finds the roots of
$
f(x) = \prod_{i=1}^r (x-t(\mathbf{x}_i)),
$
where $r$ is the number of roots of the system. Unfortunately, we find that this univariate problem can be highly ill-conditioned. 

In \cite{rouillier1999rur}, it is proved that there always exists a linear separating polynomial. Thus, a separating polynomial has the form
$
t(x_1,\ldots,x_d) = u_1x_1 + \cdots + u_dx_d,
$
 where $u_i \in \mathbb{C}$. In addition, a scale assumption is necessary to obtain meaningful estimates for the absolute condition number. We make the particular choice $\sum_{i = 1}^d \lvert u_i \rvert^2 \leq 1$.

\subsection{A Devastating Example for the Rational Univariate Representation}

We focus on the following example for this section.

\begin{ex}\label{ex:Hypercube}
For any invertible matrix $A \in \mathbb{C}^{d \times d}$ and $c > 0$, consider
$$
p_i(x_1,\ldots,x_d) = 
\sum_{j = 1}^d a_{i j} \left( x_j^2 - \frac{1}{c^2d} \right), \quad 1 \leq i \leq d,
$$
where $a_{ij}$ is the $(i,j)$ entry of $A$. This system has $2^d$ roots at $(\pm \frac{1}{c \sqrt{d}},\cdots, \pm \frac{1}{c \sqrt{d}})$.
The Jacobian at 
$\mathbf{x}^* = (\frac{1}{c \sqrt{d}},\cdots, \frac{1}{c \sqrt{d}})$ is
$
J(\mathbf{x}^*) = \frac{2}{c \sqrt{d}} A,
$
so
the absolute condition number of $\mathbf{x}^*$ is
$
\kappa_{\text{root}}(\mathbf{x}^*) = \frac{c \sqrt{d}}{2||A||_2}.
$
We prove that for any bounded linear projection of this system, the root $x^* = t(\mathbf{x}^*)$ corresponding to $\mathbf{x}^*$ has a condition number that is exponentially greater than $\kappa_{\text{root}}(\mathbf{x}^*)$.
\end{ex}

\begin{theorem}
Consider solving the system in \cref{ex:Hypercube} via the rational univariate representation using a linear projection $t(x_1,\ldots,x_d) = u_1x_1 + \cdots + u_dx_d$, with $\sum_{i=1}^d \lvert u_i \rvert^2 \leq 1$, and let $f(x) = \prod_{i=1}^r (x-t(\mathbf{x}_i))$. Then, the root $x^* = t(\mathbf{x}^*)$ of $f$ corresponding to $\mathbf{x}^*$ has condition number $\kappa_{\text{uni}}(x^*) \geq \left( \frac{c}{2} \right)^{2^d-1}$.
\end{theorem}

\begin{proof}
    Suppose we have a projection $t(x_1,\ldots,x_d)=u_1x_1 + \cdots + u_d x_d$ represented by a vector $\mathbf{u} = (u_1,\ldots,u_d)$. The one-dimensional rootfinding problem is
    $$
    f(x) = \prod_{S \subseteq [d]} \left( x - \frac{1}{c \sqrt{d}} \sum_{i=1}^d (-1)^{[i \in S]} u_i \right),
    $$
    with $[d] = \{1,\ldots,d\}$. Differentiating $f(x)$ shows that the condition number of the root $x^* = \frac{1}{c \sqrt{d}} \sum_{i=1}^d u_i$ satisfies
    \begin{align*}
    \kappa_{\text{uni}}(x^*)^{-1} &= \left\lvert \prod_{\substack{S \subseteq [d] \\ S \neq \emptyset}} \left( x^* - \frac{1}{c \sqrt{d}} \sum_{i=1}^d (-1)^{[i \in S]} u_i \right) \right\rvert
    = \left\lvert\prod_{\substack{S \subseteq [d] \\ S \neq \emptyset}} \sum_{i \in S}\frac{2u_i}{c\sqrt{d}} \right\rvert
    = \left(\frac{2}{c\sqrt{d}} \right)^{2^d-1} \left\lvert\prod_{\substack{S \subseteq [d] \\ S \neq \emptyset}} \sum_{i \in S} u_i \right\rvert \\
    &\leq \left(\frac{2}{c\sqrt{d}} \right)^{2^d-1} \prod_{\substack{S \subseteq [d] \\ S \neq \emptyset}} \sum_{i \in S} \left\lvert u_i \right\rvert 
    \leq \left(\frac{2}{c\sqrt{d}} \right)^{2^d-1}\left(\sqrt{d} \right)^{2^d-1} = \left(\frac{2}{c} \right)^{2^d-1},
    \end{align*}
    where the last inequality follows from \cref{lem:RURbound}. Thus $\kappa_{\text{uni}}(x^*) \geq \left( \frac{c}{2} \right)^{2^d-1}$.
\end{proof}

Thus, for the system in \cref{ex:Hypercube}, the root $\mathbf{x}^* = (\frac{1}{c \sqrt{d}},\cdots, \frac{1}{c \sqrt{d}})$ of the original system has $\kappa_{\text{root}}(\mathbf{x}^*) = \frac{c \sqrt{d}}{2||A||_2}$ but, if we try to solve this system using the rational univariate representation, the condition number of the corresponding root $x^* = t(\mathbf{x}^*)$ has $\kappa_{\text{uni}}(x^*) \geq \left( \frac{c}{2} \right)^{2^d-1}$. When $c>2$, this suggests a loss in accuracy due to the univariate rootfinding problem that is exponential in $d$, which aligns with practical experiments when $d=2$ (see \cref{fig:devex} \textbf{(d)}). This instability is basis independent and can not be avoided by representing $f$ in a different polynomial basis.

Choosing a nonlinear separating polynomial is also possible, which complicates the conditioning analysis. For each root $\mathbf{x}$, there exists a Lagrange interpolant $q_{\mathbf{x}}$ such that $q_{\mathbf{x}}(\mathbf{x}) = 1$ and $q_{\mathbf{x}}(\mathbf{x}') = 0$ for all roots $\mathbf{x}' \neq \mathbf{x}$. The degree of these interpolants can be bounded above (see \cref{sec:MS}). Thus, if the degree of $t$ is allowed to be large enough, then $t$ can theoretically be chosen so that $f$ is any univariate polynomial with $r$ roots, so, in particular, the roots of $f$ could be as well-conditioned as the original roots. However, we are unaware of a method for choosing such a clever separating polynomial. We suspect that finding such a separating polynomial is essentially as hard as finding the roots themselves, and as such, having a linear separating polynomial is a reasonable assumption. Still, the possibility of a more creative choice of separating polynomial is a potential way around this point of instability.

\section{Multiparameter Eigenproblem Solved by Operator Determinants} \label{sec:MEP}
Another method for solving a multidimensional rootfinding problem is to convert it into a multiparameter eigenproblem \cite{plestenjak2016roots,boralevi2017uniform,plestenjak2017minimal}. Given the polynomial system in \cref{eq:polysystem}, we first construct a determinantal representation for each $p_i$, which is a linear matrix polynomial $W_i(x_1,\ldots,x_d)$ such that $\det(W_i) = p_i$. This gives a multiparameter eigenproblem of the form:
\begin{equation} \label{eq:MEP}
W_i(\mathbf{x}) \mathbf{v}_i = \left(V_{i0} - \sum_{j=1}^d x_j V_{ij} \right) \mathbf{v}_i = 0, \quad 1 \leq i \leq d,
\end{equation}
with $V_{ij} \in \mathbb{C}^{n_i \times n_i}$ for integers $n_1,\ldots,n_d$. A solution of the multiparameter eigenproblem consists of an eigenvalue $(x_1^*,\ldots,x_d^*)$ and eigenvectors $\mathbf{v}_1,\ldots,\mathbf{v}_d$ that solve \cref{eq:MEP}. A $d$-tuple $(x_1^*,\ldots,x_d^*)$ is an eigenvalue of this multiparameter eigenproblem if and only if it is a root of the system in \cref{eq:polysystem}. The usual approach to solve the multiparameter eigenproblem is via operator determinants \cite{atkinson1972multieig}. One constructs
\[
\Delta_0 =
\begin{vmatrix}
    V_{11}& V_{12} & \cdots & V_{1d} \\
    V_{21}& V_{22} & \cdots & V_{2d} \\
    \vdots & \vdots & \ddots & \vdots \\
    V_{d1} & V_{d2} & \cdots & V_{dd}
\end{vmatrix},
\qquad 
\Delta_i = 
\begin{vmatrix}
    V_{11}& \cdots & V_{1,i-1} & V_{10} & V_{1,i+1} & \cdots & V_{1d} \\
    V_{21}& \cdots & V_{2,i-1} & V_{20} & V_{2,i+1} & \cdots & V_{2d} \\
    \vdots & \ddots & \vdots & \vdots & \vdots & \ddots & \vdots \\
    V_{d1} & \cdots & V_{d,i-1} & V_{d0} & V_{d,i+1} & \cdots & V_{dd}
\end{vmatrix},
\]
where the vertical bars denote taking the block determinant with multiplication replaced by Kronecker products. Then, one solves the following GEPs:
\begin{equation} \label{eq:opdetgeneig}
(\Delta_i -x_i \Delta_0)\mathbf{z}_i = 0, \quad 1 \leq i \leq d.
\end{equation}

\subsection{Conditioning of the Operator Determinants Method}

Unfortunately, we show that the GEPs in \cref{eq:opdetgeneig} can be exponentially ill-conditioned. Similarly to \cite{hochstenbach2003conditioning}, we relate the conditioning of this problem to the matrix
$$
B_0 = \begin{pmatrix}
    \mathbf{u}_1^{\top} V_{11} \mathbf{v}_1 & \mathbf{u}_1^{\top} V_{12} \mathbf{v}_1 & \cdots & \mathbf{u}_1^{\top} V_{1d} \mathbf{v}_1 \\
    \mathbf{u}_2^{\top} V_{21} \mathbf{v}_2 & \mathbf{u}_2^{\top} V_{22} \mathbf{v}_2 & \cdots & \mathbf{u}_2^{\top} V_{2d} \mathbf{v}_2 \\
    \vdots & \vdots & \ddots & \vdots \\
    \mathbf{u}_d^{\top} V_{d1} \mathbf{v}_d & \mathbf{u}_d^{\top} V_{d2} \mathbf{v}_d & \cdots & \mathbf{u}_d^{\top} V_{dd} \mathbf{v}_d
\end{pmatrix},
$$
with $\mathbf{u}_i,\mathbf{v}_i$ the left and right eigenvectors associated with an eigenvalue $(x_1^*,\ldots,x_d^*)$. We assume the eigenvectors are all normalized so that $||\mathbf{u}_i||_2 = ||\mathbf{v}_i||_2 = 1$. We first prove the following extension of \cite[Proposition 13]{hochstenbach2003conditioning}, which relates $B_0$ to the Jacobian of \cref{eq:polysystem}.

\begin{proposition} \label{prop:Bormultieigcond}
Let $W = (W_1,\ldots,W_d)$ be a multiparameter eigenproblem with simple eigenvalue $\mathbf{x}^* = (x^*_1,\ldots,x^*_d)$, where $W_i(\mathbf{x}) \in \mathbb{C}^{n_i \times n_i}$. Then
\begin{equation} \label{eq:B0Jacobian}
B_0 = \begin{pmatrix}
    \pm \prod_{j=1}^{n_1-1} \sigma_j^{(1)}(\mathbf{x}^*) & 0 & 0\\ 
    0 & \ddots & 0 \\
    0 & 0 & \pm \prod_{j=1}^{n_d-1}\sigma_j^{(d)}(\mathbf{x}^*)
\end{pmatrix}^{-1}
\begin{pmatrix}
    \frac{\partial p_1}{\partial x_1}(\mathbf{x}^*) & \cdots & \frac{\partial p_1}{\partial x_d}(\mathbf{x}^*)\\
    \vdots & \ddots & \vdots \\
    \frac{\partial p_d}{\partial x_1}(\mathbf{x}^*) & \cdots & \frac{\partial p_d}{\partial x_d}(\mathbf{x}^*)
\end{pmatrix},
\end{equation}
with $p_i = \det(W_i)$ and $\{\sigma_j^{(i)} \}_{j=1}^{n_i-1}$ the nonzero singular values of $W_i(\mathbf{x}^*)$.
\end{proposition}

\begin{proof}
    The proof follows the argument in \cite[Proposition 13]{hochstenbach2003conditioning}.  Note that $g(t) = \det(Z(t))$ is an analytic function of $t$,  where $Z(t) = W_i(x^*_1,\cdots,x^*_{j-1},t,x^*_{j+1},\cdots,x^*_d)$, so we can calculate its derivative along any path. We choose $t$ real. Then let $Z(t) = U(t)S(t)V(t)^H$, where $U$ and $V$ are unitary, $V(t)^H$ denotes the conjugate transpose of $V(t)$, $S$ is a real diagonal matrix, and each term on the right-hand side is an analytic function of $t$ for some small interval around $x^*_j$. The decomposition can be permuted so that $S(x^*_j) = \diag\{\sigma_k(W_i(\mathbf{x}^*))\}_{k=1}^{n_i}$. It is a simple extension of \cite[Theorem 1]{bunse1991analyticsvd} that such a decomposition exists (see \cref{app:Analytic SVD}). This is a bit like an analytic SVD, except that the diagonal entries of $S(t)$ may not be ordered.

    Denote the last columns of $U(t)$, $V(t)$ by $\mathbf{u}_{n_i}(t)$, $\mathbf{v}_{n_i}(t)$, and the $(n_i,n_i)$ entry of $S(t)$ by $\sigma_{n_i}(t)$. Then $\sigma_{n_i}(x^*_j) = 0$, $\mathbf{u}_{n_i}(x^*_j)^H = \mathbf{u}_i^{\top}$, $\mathbf{v}_{n_i}(x^*_j) = \mathbf{v}_i$, and $\sigma_{n_i}(t) = \mathbf{u}_{n_i}(t)^{H} Z(t) \mathbf{v}_{n_i}(t)$, so
    $$
    \frac{d \sigma_{n_i}}{dt}(x^*_j) = -\mathbf{u}_i^{\top} V_{ij} \mathbf{v}_i = -(B_0)_{ij}.
    $$
    Then note that $g(t) = \det(Z(t)) =  \mp \prod\sigma_i(t) $, so that
    $$
    \frac{\partial p_i}{\partial x_j}(\mathbf{x}^*) = \frac{dg}{dt} (x^*_j) = \pm \prod_{k=1}^{n_i-1} \sigma_k^{(i)}(\mathbf{x}^*)(B_0)_{ij},
    $$
    and the result follows.
\end{proof}

Note that without the assumption of a simple eigenvalue, the diagonal matrix in \cref{eq:B0Jacobian} is not invertible, but we still have
\[
\begin{pmatrix}
    \pm \prod_{j=1}^{n_1-1} \sigma_j^{(1)}(\mathbf{x}^*) & 0 & 0\\ 
    0 & \ddots & 0 \\
    0 & 0 & \pm \prod_{j=1}^{n_d-1} \sigma_j^{(d)}(\mathbf{x}^*)
\end{pmatrix} B_0 =
\begin{pmatrix}
    \frac{\partial p_1}{\partial x_1}(\mathbf{x}^*) & \cdots & \frac{\partial p_1}{\partial x_d}(\mathbf{x}^*)\\
    \vdots & \ddots & \vdots \\
    \frac{\partial p_d}{\partial x_1}(\mathbf{x}^*) & \cdots & \frac{\partial p_d}{\partial x_d}(\mathbf{x}^*)
\end{pmatrix},
\]
which shows that $\mathbf{x}^*$ is a simple root of \cref{eq:polysystem} only if $\mathbf{x}^*$ is a simple eigenvalue of the corresponding multiparameter eigenproblem. Thus, the system in \cref{eq:polysystem} always produces multiparameter eigenproblems that satisfy the conditions of \cref{prop:Bormultieigcond}, and in particular $W_i(\mathbf{x}^*)$ always has  $n_i-1$ nonzero singular values. 

We can now characterize the conditioning of the GEPs in \cref{eq:opdetgeneig}.

\begin{theorem} \label{thm:multipareig}
Let $p = (p_1,\ldots,p_d)$ be a zero-dimensional polynomial system in \cref{eq:polysystem} with all simple roots and no roots at infinity. Let $W = (W_1,\ldots,W_d)$ be a multiparameter eigenproblem with $\det(W_k) = p_k$. Then, the condition number of a GEP in \cref{eq:opdetgeneig} generated by operator determinants is
$$
\kappa_{\text{eig}}(x_i^*) = \lvert \det(B_0)^{-1} \rvert (1+|x_i^*|) = \frac{ \prod_{k=1}^d \prod_{j=1}^{n_k-1} \sigma_j^{(k)}(\mathbf{x}^*)}{\lvert \det(J(\mathbf{x}^*)) \rvert}(1+|x_i^*|),
$$
where $J$ is the Jacobian of $p$ and $\{\sigma_j^{(k)} \}_{j=1}^{n_k-1}$ are the $n_k-1$ nonzero singular values of $W_k(\mathbf{x}^*)$.
\end{theorem}

\begin{proof}
The eigenvectors of \cref{eq:opdetgeneig} are the Kronecker products of the eigenvectors of the original multiparameter eigenproblem, so the condition number of one of the resulting GEPs is
$$
\kappa_{\text{eig}}(x_i^*) = \frac{||\mathbf{u}_1 \otimes \cdots \otimes \mathbf{u}_d ||_2 ||\mathbf{v}_1 \otimes \cdots \otimes \mathbf{v}_d||_2}{ \lvert (\mathbf{u}_1 \otimes \cdots \otimes \mathbf{u}_d)^{\top} \Delta_0 (\mathbf{v}_1 \otimes \cdots \otimes \mathbf{v}_d) \rvert} (1+\lvert x_i^* \rvert),
$$
where $\mathbf{u}_i,\mathbf{v}_i$ are the normalized left and right eigenvectors corresponding to $\mathbf{x}^*$. For any matrices $V_1,\ldots,V_d$ with $V_i \in \mathbb{C}^{n_i \times n_i}$ a simple calculation gives
$$
    (\mathbf{u}_1^{\top} V_1 \mathbf{v}_1)  \cdots (\mathbf{u}_d^{\top} V_d \mathbf{v}_d) = (\mathbf{u}_1 \otimes \cdots \otimes \mathbf{u}_d)^{\top} (V_1 \otimes \cdots \otimes V_d) (\mathbf{v}_1 \otimes \cdots \otimes \mathbf{v}_d).
$$
Thus $(\mathbf{u}_1 \otimes \cdots \otimes \mathbf{u}_d)^{\top} \Delta_0 (\mathbf{v}_1 \otimes \cdots \otimes \mathbf{v}_d) = \det(B_0)$, and the result follows from \cref{prop:Bormultieigcond}.
\end{proof}

Now, we consider an example that demonstrates how the formula for $\kappa_{\text{eig}}(x_i^*)$ in \cref{thm:multipareig} leads to exponential instability.

\begin{ex} \label{ex:MEPdev}
Let $P$ be a $d \times d$ permutation matrix and $\sigma > 0$. For each $i$, there is an integer $j_i$ such that the $(i,j_i)$ entry of $P$ equals $1$. Consider the system in \cref{eq:polysystem} with 
    $$
    p_i(x_1,\ldots,x_d) = x_i^2 + \sigma x_{j_i}, \quad 1 \leq i \leq d.
    $$
The system has a root at $(0,\ldots,0)$ with condition number $\kappa_{\text{root}}(0,\ldots,0) = ||J^{-1}(0,\ldots,0)||_2 = \sigma^{-1}$. We can construct a determinantal representation of this system with
    $$
    W_i(x_1,\ldots,x_d) = \begin{pmatrix}
    1 & 0 \\
    0 & 1
    \end{pmatrix} x_i + \begin{pmatrix}
    0 & \sigma \\
    0 & 0
    \end{pmatrix} x_{j_i}  + \begin{pmatrix}
    0 & 0 \\
    -1 & 0
    \end{pmatrix}, \quad 1 \leq i \leq d.
    $$
    \cref{thm:multipareig} proves that the condition number of the GEP that results from the operator determinants method is $\kappa_{\text{eig}}(0) \geq \sigma^{-d} \prod_{k=1}^d \prod_{j=1}^{n_k-1} \sigma_j(V_{k0}) = \sigma^{-d}$. So, $\kappa_{\text{eig}}$ for the GEP in \cref{eq:opdetgeneig} can be greater than $\kappa_{\text{root}}$ by a factor that depends exponentially on $d$. This agrees with practical experiments (see \cref{fig:devex} \textbf{(e)} and \cref{fig:devexhighdimMEP}). As in \cref{sec:GB} and \cref{sec:RUR}, this instability is basis independent.
\end{ex}

While \cref{thm:multipareig} does not demonstrate that all determinantal representations of this system lead to exponential ill-conditioning, there is no known method that uses the factor $\prod_{k=1}^d \prod_{j=1}^{n_k-1} \sigma_j^{(k)}(\mathbf{x}^*)$ to control the conditioning of the GEP in \cref{eq:opdetgeneig}.  In fact, we believe that any such construction is likely as difficult as finding the roots themselves. 

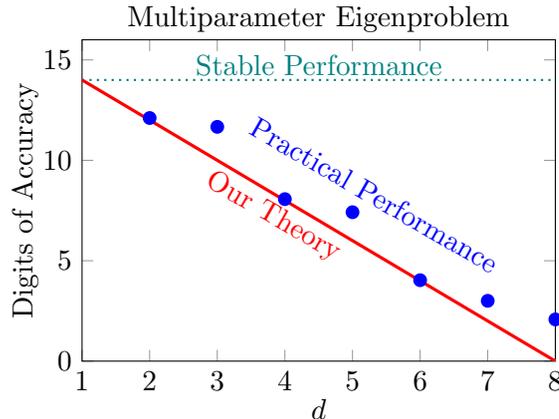
\begin{figure}
  \centering
  \begin{Overpic}{
  \begin{tikzpicture}
  %\pgfplotsset{set layers}
    \begin{axis}[
      title={Multiparameter Eigenproblem},height=2.3in,width=3.1in,xmin =1,xmax=8,ymin=0,ymax=16,
      %legend pos = north east,
      xlabel={$d$}, ylabel={Digits of Accuracy}
      ]
      \addplot [color=teal,thick,dotted,mark = none] table[col sep=comma,x = {xh}, y ={yhtMEP}] {HDdata.dat};
      \addplot [color=red,very thick,mark = none] table[col sep=comma,x = {xh}, y ={yhMEP}] {HDdata.dat};
      \addplot [color=blue,very thick,only marks] table[col sep=comma,x = {xh}, y ={S9}] {HDdata.dat};
      %\addplot [color=violet,thick,dotted] table[col sep=comma,x = {x}, y ={D2}] {D2data.dat};
       \end{axis}
       %\begin{pgfonlayer}{axis background}
       %\filldraw [fill=gray!60] (1.7,0) rectangle (2.45,4);
       %\end{pgfonlayer}
  \end{tikzpicture}}
  \put(35,61){\rotatebox{0}{\textcolor{teal}{Stable Performance}}}
  \put(44,51){\rotatebox{-30}{\textcolor{blue}{Practical Performance}}}
  \put(36,42){\rotatebox{-30}{\textcolor{red}{Our Theory}}}
  \end{Overpic}
  
  \vspace{-.6cm}
  
  \caption{Performance of the multiparameter eigenproblem method on \cref{ex:MEPdev} for $d \geq 2$ and $\sigma = \frac{1}{100}$. We plot the practical performance against the theoretical performance of a stable algorithm and the prediction given by \cref{thm:multipareig}.}
  \label{fig:devexhighdimMEP}
\end{figure}

We have focused on the conditioning of the operator determinants method in this section rather than the conditioning of the original multiparameter eigenproblem, which by the analysis of \cite{hochstenbach2003conditioning} can be as well-conditioned as the original system, so our analysis does not rule out the possibility of a better method for solving multiparameter eigenproblems. However, the operator determinants method is the only global method for solving the multiparameter eigenproblem that we know. As an additional consequence, this observation suggests that there is no known stable global method to solve multiparameter eigenproblems.

%Aside from our analysis of the conditioning of the GEPs in \cref{eq:opdetgeneig}, multiparameter eigenproblems suffer from several other possible conditioning issues. In particular, the minimal representation of \cite{plestenjak2017minimal} involves solving arbitrary univariate rootfinding problems in the monomial basis, so it is unstable \cite{wilkinson1984polynomial}. On the other hand, if one does not construct a minimal representation, the resulting GEPs are singular; for more discussion of this, see \cite{plestenjak2016roots}.

\section{Normal Form Methods}\label{sec:MS}
The structure of $\mathbb{C}[x_1,\ldots,x_d]/\langle p_1,\ldots,p_d \rangle$ gives rise to another class of rootfinding techniques, known as normal form methods \cite{telen2018normalform, telen2018normalform2, mourrain2021normalform}. Given a zero-dimensional polynomial system in \cref{eq:polysystem}, $\mathbb{C}[x_1,\ldots,x_d]/\langle p_1,\ldots,p_d \rangle$ is a finite-dimensional $\mathbb{C}$-vector space with dimension equal to the number of roots of the system and the map given by multiplication by $x_i$ is a linear map \cite{telen2018normalform}. Given a basis $\mathcal{B}$ for $\mathbb{C}[x_1,\ldots,x_d]/\langle p_1,\ldots,p_d \rangle$, we can construct the matrix $M_{x_i}$ for this map; we call these M\"oller--Stetter (MS) matrices.

We can construct a basis for $\mathbb{C}[x_1,\ldots,x_d]/\langle p_1,\ldots,p_d \rangle$ consisting of Lagrange interpolants. In this basis, the matrices $M_{x_i}$ are diagonal with entries that are the coordinates of the roots. Therefore, the eigenvalues of the MS matrices in any basis $\mathcal{B}$ are the coordinates of the roots. A normal form method solves \cref{eq:polysystem} by constructing, and then solving, the MS eigenproblems
\begin{equation} \label{eq:MSEVP}
M_{x_i} \mathbf{w} = x_i \mathbf{w}, \quad 1 \leq i \leq d.
\end{equation}

\subsection{Conditioning of the M\"oller--Stetter Eigenproblems}
To analyze the conditioning of the MS eigenproblems in \cref{eq:MSEVP} we need to characterize the left and right eigenvectors of the MS matrices. For a given basis $\mathcal{B}$, the left eigenvector corresponding to a coordinate $x_i^*$ of a particular root $\mathbf{x}^*=(x_1^*,\ldots,x_d^*)$ is $\mathcal{B}(\mathbf{x}^*)$. The right eigenvector is $[q]_{\mathcal{B}}$, where $q$ is the Lagrange interpolant for $p_1,\ldots,p_d$ that vanishes at every root of $p_1,\ldots,p_d$ except for $\mathbf{x}^*$, and $[q]_{\mathcal{B}}$ denotes the expression of the interpolant in the basis $\mathcal{B}$, also known as its normal form. We drop the subscript from $q$ for ease of notation. We begin with an extension of \cite[Lemma 7.1]{parkinson2022normalform} that gives the Lagrange interpolants. We use the notation $\mathbb{C}[x_1,\ldots,x_d]^{d \times d}$ to denote the space of $d \times d$ matrices with entries in $\mathbb{C}[x_1,\ldots,x_d]$.

\begin{proposition}
Suppose $Q \in \mathbb{C}[x_1,\ldots,x_d]^{d \times d}$ and \cref{eq:polysystem} has the form
\begin{equation}
p_i = r_i(x_1,...,x_d)(x_i-x_i^*) + \sum_{j=1}^d q_{i j}(x_1,\ldots,x_d) (x_j-x_j^*), \quad 1 \leq i \leq d,
\end{equation}
with $r_i \in \mathbb{C}[x_1,\ldots,x_d]$, $q_{ij}$ the $(i,j)$ entry of $Q$, and $(x_1^*,\ldots,x_d^*) \in \mathbb{C}^d$. Then the Lagrange interpolant that vanishes at every root of $p$ except $(x_1^*,\ldots,x_d^*)$ is
\begin{equation}
q = \sum_{\imath \subseteq [d]} \det(Q_\imath) \prod_{k \in \imath} r_k,
\end{equation}
where $[d] = \{1,...,d\}$, and $Q_\imath$ is $Q$ with the $\imath$-th rows and columns removed.
\end{proposition}
\begin{proof}
See \cref{app:MSEVproof}.
\end{proof}

We can always write a system with a root at $(x_1^*,\ldots,x_d^*)$ as
\begin{equation} \label{eq:genpolyform}
\begin{pmatrix}
    p_1(x_1,\ldots,x_d) \\
    \vdots \\
    p_d(x_1,\ldots,x_d)
\end{pmatrix} = Q
\begin{pmatrix}
    x_1 - x_1^* \\
    \vdots \\
    x_d-x_d^*
\end{pmatrix}
\end{equation}
for some $Q \in \mathbb{C}[x_1,\ldots,x_d]^{d \times d}$, and the Lagrange interpolant at $(x_1^*,\ldots,x_d^*)$ is $[\det(Q)]_{\mathcal{B}}$. This form enables the following complete characterization of the conditioning of the MS eigenproblems in \cref{eq:MSEVP}.
\begin{theorem} \label{thm:highdimMS}
    Let $p = (p_1,\ldots,p_d)$ be a zero-dimensional polynomial system in \cref{eq:polysystem} with all simple roots and no roots at infinity. Let $\mathbf{x}^* = (x_1^*,\ldots,x_d^*)$ be a root and write $p$ in the form \cref{eq:genpolyform}. Let $M_{x_i}$ be a M\"oller--Stetter matrix for $p$ with respect to a basis $\mathcal{B}$. Then, the absolute condition number of the eigenvalue $x_i^*$ is
    $$
    \kappa_{\text{eig}}(x_i^*) = \frac{||[\det(Q)]_{\mathcal{B}}||_2||\mathcal{B}(\mathbf{x}^*)||_2}{\lvert \det(J(\mathbf{x}^*)) \rvert }(1 + |x_i^*|),
    $$
    where $J$ is the Jacobian of $p$.
\end{theorem}
\begin{proof}
Note that $[\det(Q)]_{\mathcal{B}}^{\top} \mathcal{B}(\mathbf{x}^*) = \det(Q(\mathbf{x}^*)) = \det(J(\mathbf{x}^*))$, so the result is a direct application of \cref{eq:geneigconditioning} with $B=I$. 
\end{proof}

We use this to demonstrate that normal form methods are exponentially unstable under a reasonable assumption on $\mathcal{B}$. In \cite{telen2018normalform, telen2018normalform2, mourrain2021normalform}, a basis is selected from a set of basis monomials of degree $< \rho = \left( \sum_{i=1}^d \deg(p_i) \right) - d + 1$. We assume that the basis is selected in this way. For any reasonable set of basis monomials, any basis $\mathcal{B}$ selected in this manner additionally satisfies $||\mathcal{B}(0,\ldots,0)||_2 \geq 1$.

\begin{ex} \label{ex:MSDev}
    Let $P$ be a $d \times d$ permutation matrix and $\sigma > 0$. For each $i$, there is an integer $j_i$ such that the $(i,j_i)$ entry of $P$ equals $1$. Consider the system in \cref{eq:polysystem} with 
    $$
   p_i(x_1,\ldots,x_d) = x_i^2 + \sigma x_{j_i}, \quad 1 \leq i \leq d.
    $$
    The system has a root at $(0,\ldots,0)$ with condition number $\kappa_{\text{root}}(0,\ldots,0) = ||J(0,\ldots,0)^{-1}||_2 = \sigma^{-1}$. We show that the condition number of the resulting MS eigenproblems in \cref{eq:MSEVP} is greater by a factor that depends exponentially on $d$.
    For this system, $\rho = d+1$, so it is possible to reduce any polynomial to its normal form using multiples $f_kp_k$ with $\deg(f_k) \leq d-1$ \cite{telen2018normalform}. In particular, $\prod_{i=1}^d x_i - \left[ \prod_{i=1}^d x_i \right]_{\mathcal{B}} \in I$ and is of degree $\leq d$ so 
    $$
    \prod_{i=1}^d x_i - \left[ \prod_{i=1}^d x_i \right]_{\mathcal{B}} = \sum_{k=1}^d f_k p_k,
    $$
    with $\deg(f_k) \leq d-1$. If the left-hand side does not vanish, then $\prod_{i=1}^d x_i$ must appear on the right-hand side. The only way that this can happen is if a term $\sigma^{-1}\left( \prod_{\substack{i=1 \\ i \neq j_{\ell}}}^d x_i \right) (x_\ell^2-\sigma x_{j_\ell})$ appears. Subtracting this term from both sides, we obtain
    $$
    -\sigma^{-1}\left(\prod_{\substack{i=1 \\ i \neq j_\ell}}^d x_i\right)(x_\ell^2) - \left[\prod_{i=1}^d x_i\right]_{\mathcal{B}} = \sum_{k=1}^d f_k' p_k.
    $$
    If the left-hand side does not vanish, then $\sigma^{-1}\left(\prod_{\substack{i=1 \\ i \neq j_\ell}}^d x_i\right)(x_\ell^2)$ must appear on the right-hand side. The only way that this can happen is if the same term $\sigma^{-1}\left(\prod_{\substack{i=1 \\ i \neq j_\ell}}^d x_i\right)(x_\ell^2-\sigma x_{j_\ell})$ appears on the right-hand side. In this way, the only possible reduction is to repeatedly substitute $\sigma^{-1} x_\ell^2$ for $x_{j_\ell}$ and then substitute back $\sigma x_{j_\ell}$ for $x_\ell^2$. As the degree of basis elements is $\leq d$, this implies that the normal form of $\prod_{i=1}^d x_i$ must be $\prod_{i=1}^d x_i$. 
    
   If $\prod_{i=1}^d x_i$ appears in the normal form of $\det(Q)- \prod_{i=1}^d x_i$, then it would be possible to reduce $\prod_{i=1}^d x_i$ to a monomial of lower degree using multiples $f_kp_k$ with $\deg(f_k) \leq d-1$, which it is not. Thus $\prod_{i=1}^d x_i$ appears in any normal form of $\det(Q)$, so $||[\det(Q)]_{\mathcal{B}}||_2 \geq 1$. By assumption, $||\mathcal{B}(0,\ldots0)||_2 \geq 1$, so $\kappa_{\text{root}}(0,\ldots,0) \geq \sigma^{-d}$.
\end{ex}

This suggests that the MS eigenproblem can introduce instability that depends exponentially on $d$, which is supported by practical experiments (see \cref{fig:devex} \textbf{(f)}).

One might hope to avoid the instability by selecting a clever basis. In particular, we know that the MS matrices in the Lagrange interpolant basis are diagonal, so the eigenproblem is trivial. However, finding the Lagrange interpolants is, in our estimation, equally as challenging as finding the roots. All known algorithms for computing MS matrices preselect basis monomials from which to choose the basis $\mathcal{B}$. Nevertheless, clever basis selection is a potential way around the instability. We address a few other interesting subtleties of normal form methods in \cref{sec:biv}.

\section{Macaulay Resultant} \label{sec:Mac}
The Macaulay resultant method forms an eigenproblem directly from the Macaulay matrix, defined in the monomial basis \cite{jonsson2005macaulay}. Consider a polynomial system in \cref{eq:polysystem}, with $h= \mu_0 + \sum_{i=1}^d \mu_i x_i$, and $\rho = \left( \sum_{i=1}^d \deg(p_i) \right) - d + 1$. For each $p_i$ (and for $h$), $M$ has rows consisting of the monomial coefficients of multiples of $p_i$ (and $h$) by all monomials of degree $\leq \rho - \deg(p_i)$ ($\leq \rho - \deg(h)$). Consider the system in \cref{ex:dev} in two variables, and let $h = \mu_0 + \mu_1 x + \mu_2 y$. Then, the Macaulay matrix in the monomial basis is
$$
M = \begin{pNiceMatrix}[first-row,first-col]
    & 1 & x & y & x^2 & xy & y^2 & x^3 & x^2y & xy^2 & y^3 \\
p_1   &  & \sigma q_{11} & \sigma q_{12} & 1 & & & & & &  \\
xp_1  & & & & \sigma q_{11} & \sigma q_{12} & & 1 & & & \\
yp_1  & & & & & \sigma q_{11} & \sigma q_{12} & & 1 & & \\
p_2   & & \sigma q_{21}& \sigma q_{22} & & & 1 & & & & \\
xp_2  & & & & \sigma q_{21} & \sigma q_{22} & & & & 1 & \\
yp_2  & & & & & \sigma q_{21} & \sigma q_{22} & & & & 1 \\
h   & \mu_0 & \mu_1 & \mu_2 & & & & & & & \\
xh  & & \mu_0 & & \mu_1 & \mu_2 & & & & &  \\
yh  & & & \mu_0 & & \mu_1 & \mu_2 & & & & \\
x^2h & & & & \mu_0 & & & \mu_1 & \mu_2 & & \\
xyh & & & & & \mu_0 & & & \mu_1 & \mu_2 & \\
y^2h & & & & & & \mu_0 & & & \mu_1 & \mu_2
\end{pNiceMatrix}.
$$

The matrix $M$ has more rows than columns. However, there exists a square submatrix of $M$ that is singular if and only if $p_1,\ldots,p_{d},h$ have a common root \cite{jonsson2005macaulay}. To obtain a square matrix, we drop some of the rows corresponding to $h$, which must be selected so that the resulting matrix is nonsingular \cite{jonsson2005macaulay}. When $d=2$, the authors in \cite{jonsson2005macaulay} provide a careful method for selecting the rows to drop. Our analysis is valid for any dropped rows that give a nonsingular submatrix. Once we have a square matrix, we set $\mu_i = \alpha_i - \lambda \beta_i$ for $\alpha_i,\beta_i$ chosen randomly. For the example above, dropping the $x^2 h$ and $y^2 h$ rows gives

\begin{center}
\resizebox{\columnwidth}{!}{
$
M(\lambda) = \begin{pNiceMatrix}[first-row,first-col]
    & 1 & x & y & x^2 & xy & y^2 & x^3 & x^2y & xy^2 & y^3 \\
p_1   &  & \sigma q_{11} & \sigma q_{12} & 1 & & & & & &  \\
xp_1  & & & & \sigma q_{11} & \sigma q_{12} & & 1 & & & \\
yp_1  & & & & & \sigma q_{11} & \sigma q_{12} & & 1 & & \\
p_2   & & \sigma q_{21}& \sigma q_{22} & & & 1 & & & & \\
xp_2  & & & & \sigma q_{21} & \sigma q_{22} & & & & 1 & \\
yp_2  & & & & & \sigma q_{21} & \sigma q_{22} & & & & 1 \\
h   & \alpha_0 - \lambda \beta_0 & \alpha_1 - \lambda \beta_1 & \alpha_2 - \lambda \beta_2 & & & & & & & \\
xh  & & \alpha_0 - \lambda \beta_0 & & \alpha_1 - \lambda \beta_1 & \alpha_2 - \lambda \beta_2 & & & & &  \\
yh  & & & \alpha_0 - \lambda \beta_0 & & \alpha_1 - \lambda \beta_1 & \alpha_2 - \lambda \beta_2 & & & & \\
xyh & & & & & \alpha_0 - \lambda \beta_0 & & & \alpha_1 - \lambda \beta_1 & \alpha_2 - \lambda \beta_2 &
\end{pNiceMatrix}.
$}   
\end{center}
This is a GEP 
\begin{equation} \label{eq:MacEVP}
M(\lambda) \mathbf{x} = (A-\lambda B) \mathbf{x} = \begin{bmatrix}
    A_1 \\
    A_2
\end{bmatrix} - 
\lambda \begin{bmatrix}
    0 \\
    B_2
\end{bmatrix}  = 0.
\end{equation}
This GEP only has $\text{rank}(B)$ standard eigenvalues; we call these finite eigenvalues.  For the Macaulay eigenproblem, the right eigenvectors corresponding to finite eigenvalues are the column labels evaluated at a root $\mathbf{x}^*$. We examine the conditioning of this eigenproblem.

\subsection{Conditioning of the Macaulay Resultant Eigenproblem}

Although the roots are extracted from the eigenvectors, we have found it easier to examine eigenvalue conditioning. Suppose we know the coordinates of a root $\mathbf{x}^* = (x_1^*,\ldots,x_d^*)$. The eigenvalue $\lambda^*$ corresponding to this root is the solution to $h(\lambda)(\mathbf{x^*}) = 0$. We assume all the randomly chosen coefficients $\alpha_i,\beta_i$ are $O(1)$, so solving this simple linear equation for $\lambda^*$ is well-conditioned. Thus, the eigenvalue condition number provides a lower bound on the condition number of finding the root $\mathbf{x}^*$ from the GEP in \cref{eq:MacEVP}. 

% Given left and right eigenvectors $\mathbf{v},\mathbf{w}$ of the generalized eigenproblem $A-\lambda B$, the eigenvalue can be recovered as $\lambda = \frac{\mathbf{v}^{\top}A\mathbf{w}}{\mathbf{v}^{\top}B\mathbf{w}}$. We assume that the eigenvectors are scaled such that ${\mathbf{v}^{\top}B\mathbf{w}} = O(1)$, so this is a stable algorithm to derive the eigenvalues from the eigenvectors. Thus the eigenvalue condition number provides a lower bound on the eigenvector conditioning.

As stated above, the right eigenvector corresponding to a root $\mathbf{x}^*$ is the evaluation of the column labels at the root $\mathbf{x}^*$. An eigenvalue is a choice of $\lambda$ for which $h(\lambda^*)$ has a common root with $p_1,\ldots,p_d$. Then, the left eigenvector expresses a combination
\begin{equation} \label{eq:Macaulaydependence}
\sum_{i=1}^{d} P_i p_i + Hh(\lambda^*)= 0,
\end{equation}
where $P_i,H \in \mathbb{C}[x_1,\ldots,x_d]$, so that in particular $Hh(\lambda^*) \in \langle p_1,\ldots,p_d \rangle$. Thus, $H h(\lambda^*)$ must vanish at all roots of the system. However, $h(\lambda^*)$ only vanishes at $\mathbf{x}^*$, so $H$ vanishes at all roots $\neq \mathbf{x}^*$. Thus, either $H \in \langle p_1,\ldots,p_d \rangle$ or $H$ is the Lagrange interpolant for the system at $\mathbf{x}^*$. If $H \in \langle p_1,\ldots,p_d \rangle$, then for all $\lambda$ there exists a combination \cref{eq:Macaulaydependence}, which implies that the eigenproblem is singular. Thus, $H$ is the Lagrange interpolant, given by the same formula as in \cref{sec:MS}. That is, if we write the system as in \cref{eq:genpolyform} for some polynomial matrix $Q$, then $H = \det(Q)$  in $\mathbb{C}[x_1,\ldots,x_d] / \langle p_1,\cdots,p_d \rangle$. 

We want to understand exactly how the equivalence in $\mathbb{C}[x_1,\ldots,x_d] / \langle p_1,\cdots,p_d \rangle$ plays out.  For this, we note that the number of finite eigenvalues of the problem $A - \lambda B$ must equal the number of roots of \cref{eq:polysystem}. This implies that the number of rows corresponding to $h$ is equal to the dimension of $\mathbb{C}[x_1,\ldots,x_d]/ \langle p_1,\ldots,p_d \rangle$. The monomials corresponding to these rows must be independent in $\mathbb{C}[x_1,\ldots,x_d] / \langle p_1,\cdots,p_d \rangle$ to have a nonsingular eigenproblem.  Hence, they are a basis $\mathcal{B}$ for $\mathbb{C}[x_1,\ldots,x_d]/ \langle p_1,\ldots,p_d \rangle$. Then $H = [\det(Q)]_{\mathcal{B}}$. We conclude the following, which is analogous to \cref{thm:highdimMS}.

\begin{theorem} \label{thm:highdimMac}
    Let $p = (p_1,\ldots,p_d)$ be a zero-dimensional polynomial system in \cref{eq:polysystem} with all simple roots and no roots at infinity. Let $\mathbf{x}^* = (x_1^*,\ldots,x_d^*)$ be a root and write $p$ in the form \cref{eq:genpolyform}. Let $M = A-\lambda B$ be a Macaulay resultant eigenproblem in \cref{eq:MacEVP} for $p$ constructed with a random linear polynomial $h$ and $\mathcal{B}$ the basis given by the rows corresponding to $h$. Then, the absolute condition number of the eigenvalue $\lambda^*$ corresponding to the root $\mathbf{x}^* = (x_1^*,\ldots,x_d^*)$ satisfies
    $$
    \kappa_{\text{eig}}(\lambda^*) \geq  \frac{||[\det(Q)]_{\mathcal{B}}||_2||V(\mathbf{x}^*)||_2}{\lvert \det(J(\mathbf{x}^*))h(\mathbf{x}^*) \rvert},
    $$
    where $J$ is the Jacobian of $p$ and $V(\mathbf{x}^*)$ is the evaluation of the column labels at $\mathbf{x}^*$.
\end{theorem}

\begin{proof}
From \cref{eq:geneigconditioning} we have
$$
\kappa_{\text{eig}}(\lambda^*) \geq \frac{||\mathbf{v}||_2||\mathbf{w}||_2}{\lvert \mathbf{v}^{\top}B\mathbf{w} \rvert},
$$
where $\mathbf{v}$ and $\mathbf{w}$ are the left and right eigenvectors. Let the system have $r$ roots. Because of the block structure of $B$, only the last $r$ entries of $\mathbf{v}$ affect $\mathbf{v}^{\top}B\mathbf{w}$, which becomes $[\det(Q)]_{\mathcal{B}}^{\top} B_1 V(\mathbf{x}^*)$, where $B_1$ is the lower nonzero block of $B$ consisting of the rows corresponding to $h$. Then $[\det(Q)]_{\mathcal{B}}^{\top} B_1$ is the coefficient vector of $[\det(Q)]_{\mathcal{B}} \cdot h$ with respect to the column labels, so $[\det(Q)]_{\mathcal{B}}^{\top} B_1 V(\mathbf{x}^*) = (\det(Q) \cdot h)(\mathbf{x}^*) = \det(J(\mathbf{x}^*))h(\mathbf{x}^*)$, and the result follows, noting $||v||_2 \geq ||[\det(Q)]_{\mathcal{B}}||_2$.
\end{proof}

We complete this section with a devastating example for the Macaulay resultant method, which is the same system as in \cref{ex:MSDev}. We assume that $||V(0,\ldots,0)||_2 \geq 1$, which holds for any reasonable polynomial basis.

\begin{ex} \label{ex:MacDev}
    Let $P$ be a $d \times d$ permutation matrix and $\sigma > 0$. For each $i$, there is an integer $j_i$ such that the $(i,j_i)$ entry of $P$ equals $1$. Consider the system in \cref{eq:polysystem} with 
    $$
   p_i(x_1,\ldots,x_d) = x_i^2 + \sigma x_{j_i} \quad 1 \leq i \leq d.
    $$
    The system has a root at $(0,\ldots,0)$ with condition number $\kappa_{\text{root}}(0,\ldots,0) = ||J(0,\ldots,0)^{-1}||_2 = \sigma^{-1}$.  By the analysis in \cref{ex:MSDev}, $||[\det(Q)]_{\mathcal{B}}||_2 \geq 1$, and, by assumption, $||V(0,\ldots,0)||_2 \geq 1$. So, if $\lambda^*$ is the eigenvalue corresponding to $(0,\ldots,0)$, then $\kappa_{\text{eig}}(\lambda^*) \geq \sigma^{-d} \lvert h(0,\ldots,0)^{-1} \rvert$.
\end{ex}

Thus,  there is an instability with a factor that depends exponentially on $d$, which is observed in our experiments (see \cref{fig:devex} \textbf{(g)}). We refine some of these results in \cref{sec:biv}.

\section{Bivariate Refinement of Conditioning Formulas for Normal Form Methods and Macaulay Resultant Matrices} \label{sec:biv}

Examine the numerator of the formulas in \cref{thm:highdimMS} and \cref{thm:highdimMac}. 
For normal form methods, if $||[\det(Q)]_{\mathcal{B}}||_2||\mathcal{B}(\mathbf{x}^*)||_2 \approx 1$,  then the condition number is $\kappa_{\text{eig}}(x_i^*) \approx \lvert \det(J(\mathbf{x}^*))^{-1} \rvert$, which aligns with results from \cite{noferini2016instability} and \cref{sec:MEP}.
Similarly, for the Macaulay resultant method, if $||[\det(Q)]_{\mathcal{B}}||_2||V(\mathbf{x}^*)||_2 \approx 1$ then $\kappa_{\text{eig}}(\lambda^*) \approx \lvert \det(J(\mathbf{x}^*))^{-1} \rvert$.
However, this does not always occur. As discussed in \cref{sec:MS}, it is in theory possible to select $\mathcal{B}$ to make the MS eigenproblem trivial; the same is true for the Macaulay resultant eigenproblem. However, all known normal form algorithms choose basis monomials from which to subselect the basis $\mathcal{B}$, and similarly all known Macaulay resultant algorithms preselect the basis monomials for the Macaulay matrix, which typically makes $||\mathcal{B}(\mathbf{x}^*)||_2= O(1)$ and $||V(\mathbf{x}^*)||_2 = O(1)$. Of more interest is the fact that $\det(Q)$ can reduce in unexpected ways in the basis $\mathcal{B}$, as demonstrated in the following example.

\begin{ex} \label{ex:notdev}
Let $A$ be a $2 \times 2$ orthogonal matrix and $\sigma > 0$. Consider
$$
\begin{pmatrix}
p_1\\
p_2
\end{pmatrix}
= 
\begin{pmatrix}
x^2 + \sigma a_{11}x + \sigma a_{12}y \\
xy + \sigma y^2 + \sigma a_{21}x + \sigma a_{22}y
\end{pmatrix}
=
\begin{pmatrix}
x + \sigma a_{11} & \sigma a_{12} \\
y + \sigma a_{21} & \sigma y + \sigma a_{22}
\end{pmatrix}
\begin{pmatrix}
x \\
y
\end{pmatrix}.
$$
Suppose that $\mathcal{B} = \{ 1,x,y,y^2 \}$. Then
$
[\det(Q)]_{\mathcal{B}}
%= \sum_{\imath \subseteq I} \det(Q_\imath) \prod_{k \in \imath} r_k
%&= x^2 + \sigma xy + \sigma(j_{22}x + j_{11}x) + \sigma^2 j_{11} y + \sigma^2 \\
= -\sigma^2y^2 + \sigma(a_{22}x-a_{12}y) - \sigma^2 a_{21}x + \sigma^2 (a_{11}-a_{22}) y + \sigma^2,
$
so $||[\det(Q)]_{\mathcal{B}}||_2 = O(\sigma)$, and the resulting eigenvalue condition number for one of the MS eigenproblems or for the Macaulay resultant eigenproblem is
$
O(\sigma^{-1})
$
while $\lvert \det(J(0,0))^{-1} \rvert = \sigma^{-2}$.
\end{ex}
Thus, the conditioning of the eigenproblems can meaningfully differ from $\lvert \det(J(\mathbf{x}^*))^{-1} \rvert$. In two dimensions, we prove that when the eigenvalue conditioning differs from $\lvert \det(J(\mathbf{x}^*))^{-1} \rvert$, other parts of the algorithms become unstable. 

To facilitate this analysis, let $\hat{M}_{\rho}$ be the submatrix of the Macaulay matrix constructed as in \cref{sec:Mac} consisting of only rows corresponding to multiples of $p_1,\ldots,p_d$. For the system in \cref{ex:dev} in two variables:
$$
\hat{M}_{\rho} = \begin{pNiceMatrix}[first-row,first-col]
    & 1 & x & y & x^2 & xy & y^2 & x^3 & x^2y & xy^2 & y^3 \\
p_1  &  & \sigma q_{11} & \sigma q_{12} & 1 & & & & & &  \\
xp_1  & & & & \sigma q_{11} & \sigma q_{12} & & 1 & & & \\
yp_1  & & & & & \sigma q_{11} & \sigma q_{12} & & 1 & & \\
p_2   & & \sigma q_{21}& \sigma q_{22} & & & 1 & & & & \\
xp_2  & & & & \sigma q_{21} & \sigma q_{22} & & & & 1 & \\
yp_2  & & & & & \sigma q_{21} & \sigma q_{22} & & & & 1 \\
\end{pNiceMatrix}.
$$
Analogously, $\hat{M}_{\rho-1}$ consists of rows that are multiples of $p_i$ by monomials of degree $\leq \rho-1-\deg(p_i)$. The next lemma helps connect the behavior of the system in \cref{ex:notdev} to $\hat{M}_{\rho}$.
\begin{lemma} \label{lem:2DMacaulaymatrix}
    Suppose that $I = \langle p_1,p_2 \rangle$ is a zero-dimensional radical ideal with no roots at infinity. Let $\rho = \deg(p_1)+\deg(p_2)-1$. Then, the rows of the Macaulay matrix $\hat{M}_{\rho-1}$ span $I_{\leq \rho -1}$, the space of polynomials in $I$ of degree $\leq \rho-1$.
\end{lemma}

\begin{proof}
    We know from \cite{telen2018normalform} and \cite[Chapt. 1, p.46]{dickenstien2005polysystem} that any polynomial in $I_{\leq \rho-1}$ can be generated by a linear combination of the rows of $\hat{M}_{\rho}$ because the rows of $\hat{M}_{\rho}$ span $I_{\leq \rho}$, so it suffices to prove that such a combination always generates a polynomial of degree exactly $\rho$ if it includes rows that are in $\hat{M}_{\rho} \backslash \hat{M}_{\rho-1}$. Let $\deg(ap_1+bp_2) < \rho$, where $a,b \in \mathbb{C}[x_1,x_2]$, $a$ has degree $\rho-\deg(p_1)$, and $b$ has degree $\rho-\deg(p_2)$. The degree $\rho$ part is
    $
    a'p_1'+b'p_2' = 0,
    $
    where $a',p_1',b',p_2'$ are the highest degree homogeneous terms of the respective polynomials. Because the ideal $\langle p_1,p_2\rangle$ has no roots at infinity, we have that $\mathcal{V}(p_1',p_2') = \emptyset$. Thus $p_1'$ and $p_2'$ are relatively prime, so $p_1' \mid b'$ and $p_2' \mid a'$, which contradicts the degree assumptions. Thus if $ap_1+bp_2$ has degree $< \rho$, then $\deg(a) < \rho-\deg(p_1)$ and $\deg(b) < \rho- \deg(p_2)$.
\end{proof}

With this, we can explain the system's behavior in \cref{ex:notdev}. Recall that for the normal form and Macaulay resultant methods, a basis $\mathcal{B}$ for $\mathbb{C}[x_1,\ldots,x_d] / \langle p_1,\ldots,p_d \rangle$ is selected from a predetermined set of basis monomials. We have the following connection.

\begin{proposition} \label{prop:2DMacSV}
    In two dimensions, with the setup as in \cref{thm:highdimMS} or \cref{thm:highdimMac}, and assuming that $p_1$ and $p_2$ are scaled so that the maximum absolute value of their coefficients is at least $1$, we have
    $
    ||[\det(Q)]_{\mathcal{B}}||_2 \geq \sigma_{\min}(\hat{M}_{\rho}),
    $
    with $\hat{M}_{\rho}$ the Macaulay matrix with respect to a set of basis monomials that contains $\mathcal{B}$.
\end{proposition}

\begin{proof}
Note that 
$$
Adj(Q)Q = 
\begin{pmatrix}
q_{22} & -q_{12} \\
-q_{21} & q_{11}
\end{pmatrix}
\begin{pmatrix}
q_{11} & q_{12} \\
q_{21} & q_{22}
\end{pmatrix}
=
\begin{pmatrix}
\det(Q) & 0 \\
0 & \det(Q)
\end{pmatrix},
$$
so $q_{22}f -q_{12}g = x_1 \det(Q)$ and $-q_{21}f +q_{11}g = x_2 \det(Q)$. These combinations are expressed as linear combinations of the rows of the Macaulay matrix $\hat{M}_{\rho}$.

As a consequence of \cref{lem:2DMacaulaymatrix}, because $\deg([\det(Q)]_{\mathcal{B}}) \leq \rho -1$ and $\deg(\det(Q)) \leq \rho -1$, we can reduce $\det(Q)$ to $[\det(Q)]_B$ using the rows of the Macaulay matrix $\hat{M}_{\rho-1}$, which are combinations $x_1^{\alpha_1}x_2^{\alpha_2}p_i$ with $|\alpha_1+\alpha_2| \leq \rho-\deg(p_i)-1$. We want to show that there is a choice of $Q$ for which $\det(Q) = [\det(Q)]_{\mathcal{B}}$. It suffices to show for any $\alpha_1,\alpha_2$ that there exists $Q'$ such that $Q' \begin{pmatrix} x_1 \\ x_2 \end{pmatrix} = \begin{pmatrix} p_1 \\ p_2 \end{pmatrix}$ and $\det(Q') = \det(Q) +x_1^{\alpha_1}x_2^{\alpha_2}p_1$. Note that $p_2 = p_2+x_1^{\alpha_1+1}x_2^{\alpha_2+1} - x_1^{\alpha_1+1}x_2^{\alpha_2+1}$. We split these new terms, giving
$$
Q' = \begin{pmatrix}
    q_{11} & q_{12} \\
    q_{21}-x_2x_1^{\alpha_1}x_2^{\alpha_2}& q_{22}+x_1x_1^{\alpha_1}x_2^{\alpha_2}
\end{pmatrix},
$$
with $\det(Q') = \det(Q) +x_1^{\alpha_1}x_2^{\alpha_2}p_1$. In addition, because $|\alpha_1+\alpha_2| \leq \rho-\deg(p_1)-1 = \deg(p_2)-2$, this does not change the degree of $p_2$. Therefore, we can choose $Q$ such that $\det(Q)$ is reduced. This, with the scale assumption on the polynomials, proves that one of the coefficient vectors of the combinations $q_{22}f -q_{12}g = x_1 \det(Q)$ and $-q_{21}f +q_{11}g = x_2 \det(Q)$ in the Macaulay matrix $\hat{M}_{\rho}$ has magnitude at least $1$, so $||[\det(Q)]_{\mathcal{B}}||_2 \geq \sigma_{\min}(\hat{M}_{\rho})$.
\end{proof}

Now we have the following refinements of \cref{thm:highdimMS} and \cref{thm:highdimMac} for $d=2$.
\begin{theorem} \label{thm:LowDimMS}
Let $p = (p_1,p_2)$ be a zero-dimensional polynomial system in \cref{eq:polysystem} with all simple roots and no roots at infinity. Let $\mathbf{x}^* = (x_1^*,x_2^*)$ be a root and write $p$ in the form \cref{eq:genpolyform}. Assume that $p_1$ and $p_2$ are scaled so that the maximum absolute value of their coefficients is at least $1$. Let $M_{x_i}: \mathbb{C}[x_1,x_2]/\langle p \rangle \to \mathbb{C}[x_1,x_2]/\langle p \rangle$ be a MS matrix corresponding to multiplication by $x_i$ in a basis $\mathcal{B}$. Then, the absolute condition number of the eigenvalue $x_i^*$ satisfies
$$
\kappa_{\text{eig}}(x_i^*) \geq \frac{\sigma_{\min}(\hat{M}_{\rho})||\mathcal{B}(x_1^*,x_2^*)||_2}{\lvert \det(J(x_1^*,x_2^*)) \rvert},
$$
where $J$ is the Jacobian of $p$ and $\hat{M}_{\rho}$ is the Macaulay matrix that is used to select $\mathcal{B}$.
\end{theorem}

\begin{proof}
It follows by substituting \cref{prop:2DMacSV} into \cref{thm:highdimMS}.
\end{proof}

\begin{theorem} \label{thm:LowDimMac}
Let $p = (p_1,p_2)$ be a zero-dimensional polynomial system in \cref{eq:polysystem} with all simple roots and no roots at infinity. Let $\mathbf{x}^* = (x_1^*,x_2^*)$ be a root and write $p$ in the form \cref{eq:genpolyform}. Assume that $p_1$ and $p_2$ are scaled so that the maximum absolute value of their coefficients is at least $1$. Let $M = A-\lambda B$ be the Macaulay matrix eigenproblem in \cref{eq:MacEVP} for $p$ constructed with a random linear polynomial $h$. Then the absolute condition number of the eigenvalue $\lambda^*$ corresponding to the root $(x_1^*,x_2^*)$ satisfies
$$
\kappa_{\text{eig}}(\lambda^*) \geq \frac{\sigma_{\min}(A_1)||V(x_1^*,x_2^*)||_2}{\vert \det(J(x_1^*,x_2^*)) h(x_1^*,x_2^*) \rvert},
$$
where $J$ is the Jacobian of $p$ and $V(x_1^*,x_2^*)$ is the evaluation of the column labels at $(x_1^*,x_2^*)$.
\end{theorem}
\begin{proof}
It follows by substituting \cref{prop:2DMacSV} into \cref{thm:highdimMac}, noting that $\hat{M}_{\rho} = A_1$.
\end{proof}

Small singular values of $\hat{M}_{\rho}$ lead to conditioning issues for the two principal numerical constructions of MS matrices in \cite{telen2018normalform, telen2018normalform2, mourrain2021normalform}. These methods divide into column space methods,  where one selects columns from $\hat{M}_{\rho}$ to make up an invertible submatrix, and null space methods, where one computes the null space of $\hat{M}_{\rho}$. Small singular values of $\hat{M}_{\rho}$ cause conditioning issues for column space methods, as any submatrix is near-singular. It also affects null space methods (see~\cref{app:SVDPerturbation}). Thus, in two variables,  by considering both the construction and the eigenproblem, the conditioning of normal form methods is still closely related to $\lvert \det(J(\mathbf{x}))^{-1} \rvert$. Practical experiments confirm this (see \cref{fig:notdevex2D} \textbf{(a)}).

In \cite{jonsson2005macaulay}, it is noted that the Macaulay eigenproblem can be reduced to eliminate the infinite eigenvalues. Given the form of \cref{eq:MacEVP}, we can find $Z$ such that $A_1Z = 0$; then the problem reduces to $(A_2Z - \lambda B_2Z) x = 0$. The problem of finding the eigenvalues for the original and reduced problem are equivalent. Moreover, the eigenvectors corresponding to finite eigenvalues give a basis for the right null space of $A_1$. As we are focused on global rootfinding, finding the eigenvectors for the original problem is equivalent to reducing (finding $Z$) and then solving the reduced problem. The reduced problem has the same eigenvalue condition number as the original, and computing $Z$ has a condition number proportional to $(\sigma_{\min}(A_1))^{-1}$ (see~\cref{app:SVDPerturbation}).  We conclude that in two dimensions, while the condition number of the eigenproblem may differ from $\lvert \det(J(\mathbf{x}^*))^{-1} \rvert$, the conditioning of the eigenvectors, from which the roots are obtained, in general, does not. We observe this behavior in practice (see \cref{fig:notdevex2D} \textbf{(b)}).

\begin{figure}
  \centering
  
  \begin{Overpic}{
  \begin{tikzpicture}
  %\pgfplotsset{set layers}
    \begin{axis}[
      title={Normal Form Method},
      height=2.3in,width=3.1in,xmin =-8,xmax=0,ymin=0,ymax=16,
      %legend pos = south east,
      xlabel={$\log (\sigma)$}, ylabel={Digits of Accuracy}
      ]
      \addplot [color=teal,thick,dotted,mark = none] table[col sep=comma,x = {x}, y ={y2}] {NDdata.dat};
      \addplot [color=red,very thick,mark = none] table[col sep=comma,x = {x}, y ={y}] {NDdata.dat};
      \addplot [color=blue,very thick,only marks] table[col sep=comma,x = {x}, y ={S1}] {NDdata.dat};
      %\addplot [color=violet,thick,dotted] table[col sep=comma,x = {x}, y ={D3}] {D2data.dat};
       \end{axis}
       %\begin{pgfonlayer}{axis background}
       %\filldraw [fill=gray!60] (1.7,0) rectangle (2.45,4);
       %\end{pgfonlayer}
  \end{tikzpicture}}
  \put(6,69){\textbf{(a)}}
  \put(31,47){\rotatebox{18.4}{\textcolor{teal}{Stable Performance}}}
  \put(22.5,40){\rotatebox{-18}{\textcolor{blue}{Practical}}}
  \put(42.5,35.5){\rotatebox{34}{\textcolor{blue}{Performance}}}
  \put(52,33){\rotatebox{34}{\textcolor{red}{Our Theory}}}
  \end{Overpic}
  \begin{Overpic}{
  \begin{tikzpicture}
  \pgfplotsset{every axis title/.style={at={(0.5,1)},above,yshift=-3pt}}
  %\pgfplotsset{set layers}
    \begin{axis}[
      title={Macaulay Resultant},
      height=2.3in,width=3.1in,xmin =-8,xmax=0,ymin=0,ymax=16,
      %legend pos = south east,
      xlabel={$\log (\sigma)$} %ylabel={Digits of Accuracy},
      ]
      \addplot [color=teal,thick,dotted,mark = none] table[col sep=comma,x = {x}, y ={y2}] {NDdata.dat};
      \addplot [color=red,very thick,mark = none] table[col sep=comma,x = {x}, y ={y}] {NDdata.dat};
      \addplot [color=blue,very thick,only marks] table[col sep=comma,x = {x}, y ={S2}] {NDdata.dat};
      %\addplot [color=violet,thick,dotted] table[col sep=comma,x = {x}, y ={D2}] {D2data.dat};
       \end{axis}
       %\begin{pgfonlayer}{axis background}
       %\filldraw [fill=gray!60] (1.7,0) rectangle (2.45,4);
       %\end{pgfonlayer}
  \end{tikzpicture}}
  \put(0,73){\textbf{(b)}}
  \put(28,50){\rotatebox{18.5}{\textcolor{teal}{Stable Performance}}}
  \put(16,41.5){\rotatebox{-17}{\textcolor{blue}{Practical}}}
  \put(37.5,35.5){\rotatebox{34}{\textcolor{blue}{Performance}}}
  \put(45,31){\rotatebox{34}{\textcolor{red}{Our Theory}}}
  \end{Overpic}
  
  \vspace{-.6cm}
  
  \caption{Performance of a normal form method and the Macaulay resultant method on \cref{ex:notdev}. We plot the practical performance against the theoretical performance of a stable algorithm. The deviation of observations from our prediction is because of the proximity of roots for small values of $\sigma$.}
  \label{fig:notdevex2D}
\end{figure}
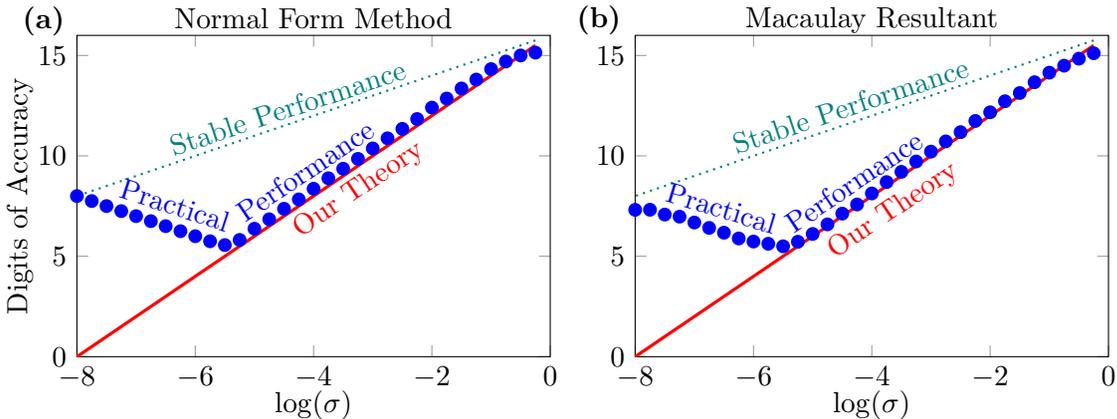

In higher dimensions, no characterization is possible in the manner of \cref{thm:LowDimMS} and \cref{thm:LowDimMac}. The following example in three dimensions can easily be generalized and demonstrates unavoidable deviation of the conditioning of the eigenproblems from $\lvert \det(J(\mathbf{x}^*))^{-1} \rvert$.

\begin{ex} \label{ex:3Ddev}
   Let
   $$
    \begin{pmatrix}
        p_1 \\
        p_2 \\
        p_3
    \end{pmatrix}
    = \begin{pmatrix}
    xy+ \sigma x^2+ \sigma y \\
    xy + \sigma y^2 + \sigma z\\
    xy + \sigma z^2 + \sigma x  \\
    \end{pmatrix} = \begin{pmatrix}
    y + \sigma x & \sigma & 0 \\
    y & \sigma y & \sigma \\
    y + \sigma & 0 & \sigma z
    \end{pmatrix}
    \begin{pmatrix}
    x \\
    y \\
    z 
    \end{pmatrix}.
    $$
    In the basis $\mathcal{B} = \{1,x,y,z,yz,xz,y^2z,xyz\}$, the Lagrange interpolant is
    $
    [\det(Q)]_{\mathcal{B}} = 
    \sigma^3 xyz + \sigma^2 y^2z + \sigma^2 y - \sigma^2yz + \sigma^3.
    $
    Then
    $
    ||[\det(Q)]_{\mathcal{B}}||_2 = O(\sigma^2),
    $
    so, for normal form methods,
    $
    \kappa_{\text{eig}}(0) = O(\sigma^{-1}).
    $
    However, in contrast to two dimensions, the smallest nonzero singular value of the Macaulay matrix $\hat{M}_{\rho}$ is only $O(\sigma^{-1})$. We can solve this in practice with an accuracy around $\sigma^{-2}$ (see \cref{fig:notdevex3D}), which beats the estimate given by $\det(J(0,0,0))^{-1} = \sigma^{-3}$. Thus,  a high-dimensional refinement of \cref{thm:highdimMS} in the manner of \cref{thm:LowDimMS} is impossible.
\end{ex}

\begin{figure}
  \centering
  
  \begin{Overpic}{
  \begin{tikzpicture}
  %\pgfplotsset{set layers}
    \begin{axis}[
      title={Normal Form Method},
      height=2.3in,width=3.1in,xmin =-8,xmax=0,ymin=0,ymax=16,
      %legend pos = south east,
      xlabel={$\log (\sigma)$}, ylabel={Digits of Accuracy}
      ]
      \addplot [color=teal,thick,dotted,mark = none] table[col sep=comma,x = {x}, y ={y2}] {NDdata.dat};
      \addplot [color=black,thick,dotted,mark = none] table[col sep=comma,x = {x}, y ={y3}] {NDdata.dat};
      \addplot [color=red,very thick,mark = none] table[col sep=comma,x = {x}, y ={y}] {NDdata.dat};
      \addplot [color=blue,very thick,only marks] table[col sep=comma,x = {x}, y ={S3}] {NDdata.dat};
      %\addplot [color=violet,thick,dotted] table[col sep=comma,x = {x}, y ={D3}] {D2data.dat};
       \end{axis}
       %\begin{pgfonlayer}{axis background}
       %\filldraw [fill=gray!60] (1.7,0) rectangle (2.45,4);
       %\end{pgfonlayer}
  \end{tikzpicture}}
  \put(31,47){\rotatebox{18.4}{\textcolor{teal}{Stable Performance}}}
  \put(20,30){\rotatebox{12}{\textcolor{blue}{Practical}}}
  \put(40,35.5){\rotatebox{34}{\textcolor{blue}{Performance}}}
  \put(51.5,32){\rotatebox{34}{\textcolor{red}{$\sigma^{-2}$}}}
  \put(55,20){\rotatebox{46}{\textcolor{black}{Jacobian Prediction}}}
  \end{Overpic}
  
  \vspace{-.6cm}
  
  \caption{Performance of a normal form method on \cref{ex:3Ddev}. We plot the practical performance against the theoretical performance of a stable algorithm, the predicted performance given by the line $\sigma^{-2}$, and the prediction given by the Jacobian, which demonstrates that this example can be solved more accurately than would be predicted by a direct analogy of \cref{thm:LowDimMS}.}
  \label{fig:notdevex3D}
\end{figure}
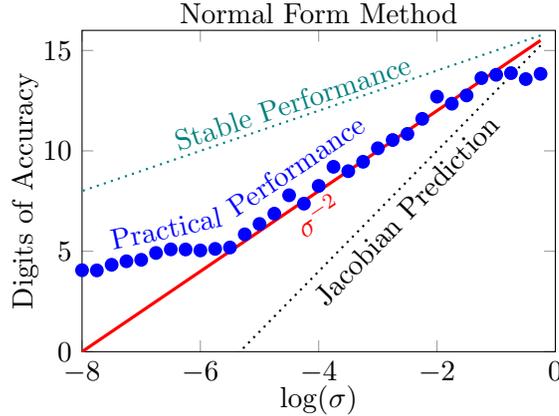

\appendix

\section{Technical Lemma for Rational Univariate Representation} \label{app:RURbound}
The following lemma allows us to bound the condition number of the univariate rootfinding problem constructed in the rational univariate representation. We use this bound to prove that solving the system in \cref{ex:Hypercube} using the rational univariate representation generates a univariate polynomial with an exponentially ill-conditioned root.

\begin{lemma} \label{lem:RURbound}
The function
$
f(u_1,\ldots,u_d) = \prod_{S \subseteq [d]} \sum_{i \in S} |u_i|
$
attains a maximum in the unit ball $\sum_{i = 1}^d \lvert u_i \rvert^2 \leq 1$ at $\mathbf{u}_0 = (d^{-1/2},\ldots,d^{-1/2})$. Moreover,
$
    f(\mathbf{u}_0) \leq (\sqrt{d})^{2^d-1}.
$
\end{lemma}

\begin{proof}
Let $\mathbf{u} = (u_1,\ldots,u_d)$ and $f_m(\mathbf{u}) = \prod_{\substack{S \subseteq [d] \\ |S| = m}} \sum_{i \in S} |u_i|$. We have
$
f(\mathbf{u}) = \prod_{m=1}^d \!f_m(\mathbf{u}),
$
so it suffices to show that $f_m(\mathbf{u})$ is maximized at $\mathbf{u}_0$ for each $m$.
Since $f_m$ is homogeneous, the maximum in the unit ball is the same as the maximum on the unit sphere. 
Suppose that there exists a point $\mathbf{v} \neq \mathbf{u}_0$ with $f_m(\mathbf{v})$ the maximum of $f_m$ on the unit sphere. We may assume $v_i \geq 0, 1 \leq i \leq d$. Because $\mathbf{v} \neq \mathbf{u}_0$, there exists $v_k \neq v_{\ell}$.  Let $v_i' = (v_k+v_\ell)/2$ if $i=k$ or $i=\ell$; otherwise,  set $v_i'=v_i$. 
%$$
%v_i' = \begin{cases}
%    (v_k+v_\ell)/2, \quad &\text{if } i=k \text{ or } i=\ell, \\
%    v_i, &\text{otherwise.}
%\end{cases}
%$$
Also,  let $
    A_m = \{S \subseteq [d] \mid |S| = m, |S \cap \{k,\ell\}| = 1\},
    B_m = \{S \subseteq [d] \mid |S| = m, |S \cap \{k,\ell\}| \neq 1\},$ and $
    D_m = \{S \subseteq [d] \mid |S| = m-1, S \cap \{k,\ell\} = \emptyset\}.
$
Then
\begin{align*}
    f_m(\mathbf{v}') 
    %&= \left( \prod_{S \in B_m} \sum_{i \in S} v_i' \right) \left( \prod_{S \in A_m} \sum_{i \in S} v_i' \right)
    &= \left( \prod_{S \in B_m} \sum_{i \in S} v_i \right) \left( \prod_{S \in A_m} \sum_{i \in S} v_i' \right)
    = \left( \prod_{S \in B_m} \sum_{i \in S} v_i \right) \left( \prod_{S \in D_m} \left( \frac{v_k + v_{\ell}}{2} + \sum_{i \in S} v_i \right)^2 \right) \\
    &> \left( \prod_{S \in B_m} \sum_{i \in S} v_i \right) \left( \prod_{S \in D_m} \left( v_k + \sum_{i \in S} v_i \right)\left( v_{\ell} + \sum_{i \in S} v_i \right) \right) \\
    &= \left( \prod_{S \in B_m} \sum_{i \in S} v_i \right) \left( \prod_{S \in A_m} \sum_{i \in S} v_i \right)
    = f_m(\mathbf{v}).
\end{align*}

\noindent Thus $f_m(\mathbf{v}') > f_m(\mathbf{v})$, so the maximum is attained at $\mathbf{u}_0$, and 
\begin{equation*}
f(\mathbf{u}_0) = \prod_{m=1}^d \left( \frac{m}{\sqrt{d}} \right)^{\binom{d}{m}} \leq  \left(\frac{d}{\sqrt{d}} \right)^{\sum_{m=1}^d \binom{d}{m}} \leq  \left(\sqrt{d} \right)^{2^d-1}.
\end{equation*}
\end{proof}

The quantity $f(\mathbf{u})$ scales the absolute condition number of a univariate polynomial constructed in the rational univariate representation; intuitively, we have used the scale restriction $\sum_{i = 1}^d \lvert u_i \rvert ^2 \leq 1$ on the projection to impose a further restriction on the univariate polynomial.

\section{An Analytic Decomposition for Complex Square Matrices}\label{app:Analytic SVD}
We use the following extension of \cite[Theorem 1]{bunse1991analyticsvd} in the proof of \cref{prop:Bormultieigcond} to relate the conditioning of the GEPs in \cref{eq:opdetgeneig} to the Jacobian of the system in \cref{eq:polysystem}.

\begin{theorem} \label{thm:analyticSVD}
    Let $Z(t)$ be a function of a complex variable $t$ such that $Z(t) \in \mathbb{C}^{d \times d}$  and $Z(t)$ is analytic for all $t$ in some real interval around a point $t_0$. Then there exists a decomposition on some (possibly distinct) real interval around $t_0$
    \begin{equation}
    Z(t) = U(t)S(t)V(t)^H,
    \end{equation}
    where $U(t),S(t),V(t)$ are analytic functions with $U(t),V(t)$ unitary, and $S(t)$ real diagonal.
\end{theorem}

\begin{proof}
The proof is similar to \cite[Theorem 1]{bunse1991analyticsvd}.
    \cite[pp. 120-122]{kato1995perturbation} gives that for some interval around $t_0$ there exists an analytic eigendecomposition
    \begin{equation}
    M(t) = \begin{bmatrix} \label{eq:analyticeigen}
        0 & Z(t) \\
        Z(t)^H & 0
    \end{bmatrix} = Q(t) \Lambda(t) Q(t)^H,
    \end{equation}
    with $Q(t)$ unitary and $\Lambda(t)$ real diagonal. The eigenvalues and eigenvectors of $M(t)$ can be paired in the following manner. If $\smash{\begin{bmatrix}
        u(t) &
        v(t)
    \end{bmatrix}^H}$
    is an eigenvector associated with an eigenvalue $\lambda(t)$, then $\smash{\begin{bmatrix}
        u(t) &
        -v(t)
    \end{bmatrix}^H}$
    is an eigenvector associated with the eigenvalue $-\lambda(t)$. Thus permuting \cref{eq:analyticeigen} gives
    $$
    Q_1(t) = \frac{1}{\sqrt{2}}\begin{bmatrix}
        U(t) & U(t) \\
        V(t) & -V(t)
    \end{bmatrix},
    \qquad
    \Lambda_1(t) = \begin{bmatrix}
        S(t) & 0 \\
        0 & -S(t)
    \end{bmatrix},
    $$
    such that
    $
    M(t) = Q_1(t)\Lambda_1(t)Q_1(t)^H,
    $
    where $S$ is real diagonal and $U$ and $V$ are unitary, which implies that
    $
    Z(t) = U(t)S(t)V(t)^H.
    $
\end{proof}
This can not properly be called an analytic SVD, as the singular values are unordered. Still, it is sufficient for the proof of \cref{prop:Bormultieigcond}.

\section{Proof of the Form of Lagrange Interpolants} \label{app:MSEVproof}
To understand the conditioning of the MS and Macaulay eigenproblems we need to understand their eigenvectors. For both methods, either the left or the right eigenvector for an eigenvalue corresponding to a root $(x_1^*,\ldots,x_d^*)$ is related to the Lagrange interpolant that vanishes at every root of the system in \cref{eq:polysystem} except for $(x_1^*,\ldots,x_d^*)$. Therefore, we need a precise formula to generate this interpolant for any root. The following extends \cite[Lemma 7.1]{parkinson2022normalform}.

\begin{proposition}
Suppose $Q \in \mathbb{C}[x_1,\ldots,x_d]^{d \times d}$ and $p$ is a polynomial system in \cref{eq:polysystem} with
\begin{equation}
p_i = r_i(x_1,...,x_d)(x_i-x_i^*) + \sum_{j=1}^d q_{i j}(x_1,\ldots,x_d) (x_j-x_j^*), \quad 1 \leq i \leq d,
\end{equation}
with $r_i \in \mathbb{C}[x_1,\ldots,x_d]$, $q_{ij}$ the $(i,j)$ entry of $Q$, and $(x_1^*,\ldots,x_d^*) \in \mathbb{C}^d$. Then the Lagrange interpolant that vanishes at every root of $p$ except $(x_1^*,\ldots,x_d^*)$ is
\begin{equation}
q = \sum_{\imath \subseteq [d]} \det(Q_\imath) \prod_{k \in \imath} r_k,
\end{equation}
where $[d] = \{1,...,d\}$ and $Q_\imath$ is $Q$ with the $\imath$-th rows and columns removed.
\end{proposition}
\begin{proof}
To see this, we show that
\begin{equation} \label{xq1}
x_i q = \sum_{j = 1}^d p_j \left( \sum_{\imath \subseteq [d]\backslash \{i,j\}} \cof_{ji} (Q_{\imath}) \prod_{k \in \imath} r_k \right),
\end{equation}
where $\cof_{ji} (Q_{\imath})$ denotes the cofactor of $Q$ obtained by removing rows $\imath \cup \{j\}$ and columns $\imath \cup \{i\}$ from Q.  The right-hand side of \cref{xq1} splits as
$$
 \sum_{\ell=1}^d \left[ r_\ell x_\ell \left( \sum_{\imath \subseteq [d]\backslash \{i,\ell\}} \cof_{\ell i} (Q_{\imath}) \prod_{k \in \imath} r_k \right) +  \sum_{j = 1}^d q_{j \ell}x_\ell  \left( \sum_{\imath \subseteq [d]\backslash \{i,j\}} \cof_{ji} (Q_{\imath}) \prod_{k \in \imath} r_k \right) \right].
$$
When $\ell = i$, this becomes
\begin{align*}
    & \left[ r_i x_i \left( \sum_{\imath \subseteq [d]\backslash \{i\}} \cof_{i i} (Q_{\imath}) \prod_{k \in \imath} r_k \right) +  \sum_{j = 1}^d q_{j i}x_i  \left( \sum_{\imath \subseteq [d]\backslash \{i,j\}} \cof_{ji} (Q_{\imath}) \prod_{k \in \imath} r_k \right) \right] \\
    %&\qquad = x_i \left[ r_i \left( \sum_{\imath \subseteq [d]\backslash \{i\}} \cof_{i i} (Q_{\imath}) \prod_{k \in \imath} r_k \right) +  \sum_{j = 1}^n q_{j i}  \left( \sum_{\imath \subseteq [d]\backslash \{i,j\}} \cof_{ji} (Q_{\imath}) \prod_{k \in \imath} r_k \right) \right] \\
    %&\qquad = x_i \left[ \left( \sum_{\substack{\imath \subseteq [d]\\ i \in \imath}} \det(Q_{\imath}) \prod_{k \in \imath} r_k \right) +  \sum_{j = 1}^n q_{j i}  \left( \sum_{\imath \subseteq [d]\backslash \{i,j\}} \cof_{ji} (Q_{\imath}) \prod_{k \in \imath} r_k \right) \right] \\
    %&\qquad  \text{(making new left-hand index sets include i)} \\
    &\qquad = x_i \left[ \left( \sum_{\substack{\imath \subseteq [d]\\ i \in \imath}} \det(Q_{\imath}) \prod_{k \in \imath} r_k \right) +  \left( \sum_{\imath \subseteq [d]\backslash \{i\}} \sum_{j \notin \imath} q_{j i}  \cof_{ji} (Q_{\imath}) \prod_{k \in \imath} r_k \right) \right]. \\
    %&\qquad  \text{(swapping order of right-hand sums)} \\
    & \hspace{-.2cm} \text{The inner sum of the right-hand term is a cofactor expansion of } \det(Q_{\imath}), \text{ so we have} \\
    &\qquad = x_i \left[ \left( \sum_{\substack{\imath \subseteq [d]\\ i \in \imath}} \det(Q_{\imath}) \prod_{k \in \imath} r_k \right) +  \left( \sum_{\imath \subseteq [d]\backslash \{i\}} \det(Q_{\imath}) \prod_{k \in \imath} r_k \right) \right]
    %&\qquad  \text{(simplifying the \cofactor expansion)} \\
    = x_i\sum_{\imath \subseteq [d]} \det(Q_{\imath}) \prod_{k \in \imath} r_k = x_i q.
\end{align*}
\vspace{-.2cm}
Therefore, it suffices to show that the terms of \cref{xq1} vanish when $\ell \neq i$. In this case
\begin{align*}
    & \left[ r_\ell x_\ell \left( \sum_{\imath \subseteq [d]\backslash \{i,\ell\}} \cof_{\ell i} (Q_{\imath}) \prod_{k \in \imath} r_k \right) +  \sum_{j = 1}^d q_{j \ell}x_\ell  \left( \sum_{\imath \subseteq [d]\backslash \{i,j\}} \cof_{ji} (Q_{\imath}) \prod_{k \in \imath} r_k \right) \right] \\
    %&\qquad = x_\ell \left[ r_\ell \left( \sum_{\imath \subseteq [d]\backslash \{i, \ell\}} \cof_{\ell i} (Q_{\imath}) \prod_{k \in \imath} r_k \right) +  \sum_{j = 1}^n q_{j \ell}  \left( \sum_{\imath \subseteq [d]\backslash \{i,j\}} \cof_{ji} (Q_{\imath}) \prod_{k \in \imath} r_k \right) \right] \\
    &\qquad = x_\ell \left[ r_\ell \left( \sum_{\imath \subseteq [d]\backslash \{i, \ell\}} \cof_{\ell i} (Q_{\imath}) \prod_{k \in \imath} r_k \right) +  \left( \sum_{\imath \subseteq [d]\backslash \{i\}} \sum_{j \notin \imath} q_{j \ell}  \cof_{ji} (Q_{\imath}) \prod_{k \in \imath} r_k \right) \right]. \\
    %&\qquad  \text{(swapping order of right-hand sums)} \\
    %&\qquad = x_\ell \left[ r_\ell \left( \sum_{\imath \subseteq [d]\backslash \{i, \ell\}} \cof_{\ell i} (Q_{\imath}) \prod_{k \in \imath} r_k \right) +  \left( \sum_{\substack{\imath \subseteq [d]\backslash \{i\} \\ \ell \in \imath}} \sum_{j \notin \imath} q_{j \ell}  \cof_{ji} (Q_{\imath}) \prod_{k \in \imath} r_k \right) \right] \\
    & \hspace{-.8cm} \text{If $\ell \notin \imath$ then the right-hand sum vanishes, so we can factor out $r_{\ell}$ to obtain} \\
    &\qquad = x_\ell r_\ell \left[ \left( \sum_{\imath \subseteq [d]\backslash \{i, \ell\}} \cof_{\ell i} (Q_{\imath}) \prod_{k \in \imath} r_k \right) +  \left( \sum_{\imath \subseteq [d]\backslash \{i, \ell\} } \sum_{j \notin \imath} q_{j \ell}  \cof_{ji} (Q_{\imath \cup \{ \ell\}}) \prod_{k \in \imath} r_k \right) \right].
    %&\qquad  \text{(factoring out $r_\ell$)} \\
    %&\qquad = x_\ell r_\ell \left[ \left( \sum_{\imath \subseteq [d]\backslash \{i, \ell\}} \cof_{\ell i} (Q_{\imath}) \prod_{k \in \imath} r_k \right) -  \left( \sum_{\imath \subseteq [d]\backslash \{i, \ell\} } \cof_{\ell i} (Q_{\imath }) \prod_{k \in \imath} r_k \right) \right] \\
    %&\qquad  \text{(noting the second sum is the expansion of the $\ell,i$ \cofactor along column $\ell$)} \\
    %&\qquad = 0.
\end{align*}
Now note that $\sum_{j \notin \imath} q_{j \ell}  \cof_{ji} (Q_{\imath \cup \{ \ell\}})$ is the cofactor expansion of $-\cof_{\ell i}Q_\imath$ along column $\ell$ so the entire expression vanishes. This proves the desired claim.
\end{proof}
We use the characterization in the proofs of \cref{thm:highdimMS}, \cref{thm:highdimMac}, \cref{thm:LowDimMS}, and \cref{thm:LowDimMac} to relate the conditioning of the eigenproblem to the Jacobian of \cref{eq:polysystem}.

\section{Perturbation Theory for the SVD} \label{app:SVDPerturbation}
In \cref{sec:biv}, we refine our characterization of the condition number of the MS and Macaulay resultant eigenproblems by examining the singular values of the Macaulay matrix $\hat{M}_{\rho}$. We claim that small singular values make the calculation of its null space unstable. Given a matrix $A$, partition the SVD
$$
A = \begin{bmatrix}
    U_s & U_{\perp}
\end{bmatrix}
\begin{bmatrix}
    \Sigma_s & \\
    & 0
\end{bmatrix}
\begin{bmatrix}
    V_s^H \\
    V_{\perp}^H
\end{bmatrix}.
$$
Then let $\tilde{A} = A+N$ where $N$ is a perturbation of norm $\epsilon$. From \cite{vaccaro1994svd}, an orthonormal basis for the perturbed null space $\tilde{U}_{\perp}$ is given by
$
(U_{\perp} + U_sQ)(I+QQ^H)^{-1/2}.
$
A first-order approximation for $Q$ is
$
Q \stackrel{_1}{=} -\Sigma_s^{-1} V_s^H N^H U_{\perp},
$
where the symbol $\stackrel{_1}{=}$ denotes equality up to first order.
As in \cite{stewart1973perturbation}, such an expansion can be connected to the gap $\gamma(U,U_{\perp})$, where
$$
\gamma(X,Y) = \max \left\{ \sup_{\substack{||x|| =1 \\ x \in X}} \inf_{y \in Y} ||x-y||,\sup_{\substack{||y|| =1 \\ y \in Y}} \inf_{x \in X} ||y-x|| \right\}.
$$
In particular, \cite[pp. 735-736]{stewart1973perturbation} gives us that the cosines of the canonical angles between $U_{\perp}$ and $\Tilde{U}_{\perp}$ are the singular values of
$
    (I+QQ^H)^{-1/2} \stackrel{_1}{=} (I+\epsilon^2\Sigma_s^{-2})^{-1/2}.
$
In particular
$
\cos(\theta_1) \stackrel{_1}{=} (1+\epsilon^2 \sigma_{\min}^{-2})^{-1/2} \stackrel{_1}{=} (1-\epsilon^2 \sigma_{\min}^{-2})^{1/2}.
$
So, by \cite[Corollary 2.6]{stewart1973perturbation},
$\gamma(U,U_{\perp}) = \sin(\theta_1) \stackrel{_1}{=} \epsilon \sigma_{\min}^{-1}$.
This analysis shows that small singular values in $\hat{M}_{\rho}$ make null space calculations of MS matrices unstable and worsen the eigenvector conditioning of the Macaulay resultant eigenproblem, which aligns with practical results in \cref{fig:notdevex2D}.

\section*{Acknowledgments}
We thank Vanni Noferini for many discussions over the years on the subject of this manuscript.   We also thank Sujit Rao, who as an undergraduate at Cornell, began to look at the numerical stability of algebraic geometric rootfinders.  In particular, the stability argument for the rational univariate representation method was sketched out by him.

\bibliographystyle{siamplain}
\bibliography{InstabilityofRootfinders}
\end{document}